\documentclass[12pt]{article}
\usepackage{amsmath, amssymb, amsthm}
\usepackage{geometry}
\usepackage{graphicx}
\usepackage{subcaption}
\usepackage{hyperref}
\usepackage{graphicx}
\usepackage{booktabs}
\usepackage{subcaption} 
\usepackage[numbers]{natbib}
\usepackage{hyperref}

\newtheorem{theorem}{Theorem}[section]
\newtheorem{lemma}[theorem]{Lemma}

\theoremstyle{definition}

\newtheorem{remark}[theorem]{Remark}
\newcommand{\doilink}[1]{\href{#1}{#1}}

\title{Necessary and Sufficient Conditions for Existence of a Unique Solution to Gamma Moment Closure for the Stochastic Ricker Equation}

\author{
Haiyan Wang \\ 
School of Mathematical and Natural Sciences, Arizona State University \\ 
Phoenix, AZ 85069, USA \\ 
\texttt{haiyan.wang@asu.edu}
\and
Melinda Wang \\ 
Circana \\ 
Chicago, IL 60601, USA \\ 
\texttt{melindaw@alumni.cmu.edu}
}

\date{}

\begin{document}
\maketitle

\begin{abstract}
This paper investigates the stochastic Ricker difference equation $X_{n+1} = X_n \exp(r(1-X_n)) \varepsilon_n$, where $X_n$ is a random variable representing the population size and $\{\varepsilon_n\}$ denotes independent random perturbations with $E[\varepsilon_n] = 1$ and $E[\varepsilon_n^2] = v > 1$. We derive a closed system of difference equations for the mean and variance of $X_n$ using the Gamma moment-closure technique and numerically verify the validity of the Gamma moment-closure approximation. By constructing an auxiliary function, we establish the necessary and sufficient condition, $v < (2 - e^{-r})^2$, for the existence of the positive unique feasible equilibrium. We further verify its local stability with numerical analysis. Monte Carlo simulations confirm the validity of the Gamma moment approximation and illustrate how the interplay between the intrinsic growth rate $r$ and noise intensity $v$ determines population persistence. The results provide a unified theoretical framework for analyzing stochastic Gamma dynamics, offering new biological insights into the stabilizing and destabilizing effects of environmental variability.
\end{abstract}

\textbf{keywords:} stochastic Ricker population model, difference equation, Gamma moment-closure, equilibrium, existence and uniqueness, stability    

\section{Introduction}\label{in}
The Ricker difference equation is a seminal model in ecology and population biology, particularly valued for describing species with non-overlapping generations, such as certain fish stocks and insect populations. It mathematically formalizes density-dependent regulation through an exponential feedback mechanism, which realistically prevents the population from becoming negative even during extreme crashes. By capturing the crucial balance between biological reproductive potential and environmental carrying capacity, the Ricker map serves as a cornerstone for analyzing complex nonlinear behaviors, ranging from stable equilibria to chaotic oscillations in resource-limited systems. The qualitative behavior of the model is highly sensitive to the intrinsic growth rate parameter, which determines the severity of density-dependent regulation and governs transitions between stable equilibria, periodic cycles, and chaotic dynamics \cite{ricker1954stock,murray2002,may1976simple,kot2001elements,strogatz2018nonlinear,Koetke2020}.

The integration of stochasticity into the Ricker model represents a critical advancement over the classical deterministic framework, shifting the focus from idealized trajectories to the probabilistic realities of ecological systems. By embedding random environmental perturbations directly into the growth dynamics, these extended models acknowledge that population driving forces are rarely constant but rather fluctuate in response to climate shifts and resource unpredictability. This approach replaces the singular prediction of a deterministic map with a distribution of potential futures, allowing researchers to quantify the variance and risks inherent in population trends. Accounting for such variability is critical in ecological modeling, where stochastic forcing can disrupt stable equilibria and drive populations toward instability or extinction. Consequently, the study of stochastic difference equations has evolved into a vibrant domain of research, offering essential tools for developing robust conservation strategies and sustainable harvesting policies under uncertainty \cite{allen2010,Arnold1998,kot2001elements,Schreiber2021,Braverman2013,Yan2024,Hognas1997,Kornadt1991}.

Analytical treatment of stochastic maps is frequently obstructed by the unclosed nature of the moment equations, where the dynamic evolution of lower-order statistics is mathematically coupled to higher-order terms. To resolve this infinite hierarchy without resorting to computationally expensive Monte Carlo simulations, moment-closure schemes approximate the underlying probability density to parameterize higher moments as functions of the mean and variance. This technique effectively projects the complex, stochastic trajectory of the population onto a tractable, low-dimensional deterministic map governing the statistical parameters. By reducing the stochastic system to a set of closed-form recurrence relations, this framework allows for a rigorous exploration of stability boundaries and bifurcation structures, preserving the critical impact of environmental variability while enabling efficient algebraic analysis \cite{Kuehn2016,Makarem2017,Marrec2023,Matis1996,Du2005,Hyon2008,allen2010,Socha2008,Lakatos_2015,Taylor2012}. 

In particular, \cite{Lakatos_2015} expands the spectrum of moment-closure techniques by introducing a multivariate Gamma approximation, designed to complement standard Gaussian methods. While the Gaussian closure serves as a robust and efficient baseline for many stochastic systems, its symmetric support can be limiting when modeling populations that are strictly positive or highly skewed. The Gamma closure addresses these specific structural characteristics by naturally enforcing non-negative constraints and capturing the asymmetric fluctuations often observed in biological dynamics. By formalizing the multivariate extension of this distribution, the authors provide a framework that preserves the physical realism of the system while maintaining the computational tractability essential for analyzing complex stochastic dynamics.

In this paper, we study the stochastic Ricker difference equation
\begin{equation}
  X_{n+1} = X_n e^{r(1 - X_n)} \varepsilon_n,\, n=1,2,... \label{eq:model}
\end{equation}
where \( X_n \) is a random variable representing population size at time \( n \), $r > 0$ is a constant representing the intrinsic growth rate and $\varepsilon_n$ 
is a random variable accounting for stochastic perturbations, independent of $X_n$, 
with $E[\varepsilon_n] = 1$ and $E[\varepsilon_n^2] = v > 1$.  Under the assumption of the Gamma moment-closure approximation, we derive a system of difference equations and establish the necessary and sufficient condition ($v< \big(2-e^{-r}\big)^2$), for the existence of a unique feasible solution.  Furthermore, we numerically demonstrate that these solutions are locally stable. The study also includes the assessment of the validity of the Gamma moment approximation. These findings offer a unified framework that connects the stochastic Gamma model's parameters to population persistence, highlighting the role of environmental variability as a stabilizing force.

The study continues the work on the stochastic logistic and Ricker equations \cite{wang2025,Wang2026,wang2025chaos}, where the authors investigate the dynamics of the stochastic models \cite{wang2025}, stability and bifurcation of the stochastic logistic equation under the Gaussian moment approximation \cite{Wang2026}, and the transitions from steady states to chaos in the stochastic models \cite{wang2025chaos}. The key new contributions of this paper include:
\begin{enumerate}
    \item The development of a closed system of difference equations from the stochastic Ricker model under the Gamma moment-closure approximation, establishing a moment-based framework for analyzing stochastic dynamics;
    \item A numerical assessment of the validity of the Gamma moment-closure approximation for the stochastic Ricker equation;
    \item The establishment of the necessary and sufficient condition ($v< \big(2-e^{-r}\big)^2$) for the existence of a unique feasible equilibrium of the closed system of difference equations by constructing and analyzing an auxiliary function. Furthermore, numerical analysis demonstrates that the region of stability coincides exactly with the existence region of the unique feasible equilibrium. These results provide new biological insights into how environmental variability constrains population stability and drives extinction risk.
    \item An investigation, using Monte Carlo simulations, of the effects of stochastic perturbations and distributional parameters on the intrinsic growth rate and population persistence.  
\end{enumerate}

The remainder of this paper is organized as follows.
In Section~\ref{differnceE}, we employ the Gamma moment-closure technique to derive a closed system of difference equations governing the mean and variance dynamics of the stochastic model.
Section~\ref{feasb} establishes the mathematical conditions for the existence and uniqueness of feasible positive equilibria.
Subsequently, Section~\ref{stables} provides a numerical analysis of the stability regions for these equilibria.
Finally, Section~\ref{discuss} summarizes the key findings, discusses their biological implications for population persistence in fluctuating environments.

\section{Gamma Moment Closure and the Resulting Difference Equations} \label{differnceE}

Given the random variable \(X\) and assume it follows the Gamma distribution with shape \(k>0\) and scale \(\theta>0\),
\[
X \sim \operatorname{Gamma}(k,\theta),
\]
having probability density
\[
f_{X}(x)=\frac{x^{k-1}e^{-x/\theta}}{\Gamma(k)\,\theta^{k}},\qquad x>0,
\]
where $\Gamma(k)$ is the Gamma function and 
\[
\mu=\mathbb{E}[X]= k\theta,\qquad s=\mathrm{Var}[X]=k\theta^2
\]
and hence \(k=\mu^2/s\), \(\theta=s/\mu\) \cite{wikipedia_gamma}. It can be further proved that the identity (see Appendix of \cite{wang2025}) holds for $n=1,2,...$
\begin{equation}\label{ident}
\mathbb{E}\big[X^n e^{-\tau X}\big]=\frac{\Gamma(k+n)}{\Gamma(k)}\frac{\theta^n}{(1+\tau \theta)^{\,k+n}}, \tau>0
\end{equation}
In particular, for $n=1,2$, we have, for $\tau >0$
\begin{align}
\begin{aligned}
\mathbb{E}\big[X e^{- \tau X}\big] &= \frac{k\theta}{(1+ \tau\theta)^{\,k+1}}= \frac{\mu}{(1+ \tau\theta)^{\,k+1}}, \\
\mathbb{E}\big[X^2 e^{- \tau X}\big] &= \frac{k(k+1)\theta^2}{(1+\tau\theta)^{\,k+2}}=\frac{\mu^2+s}{(1+\tau\theta)^{\,k+2}}.
\end{aligned}
\label{r10}
\end{align}
where $ k=\frac{\mu^2}{s}, \, \theta=\frac{s}{\mu}$

To obtain a tractable system for the stochastic difference equations, we adopt the Gamma moment-closure approximation. This approach does not assume that $X_n$ itself follows the Gamma distribution,  but rather that its first few moments satisfy the same relationships as those of a Gamma random variable. 
Specifically, we assume that the moments of $X_n, n=1,2,...$, satisfy
\begin{align}
\begin{aligned}
\mu_n \qquad\qquad & = \mathbb{E}[X_n]\\
 s_n \qquad\qquad &= \mathrm{Var}(X_n) \\
\mathbb{E}\big[X_n e^{-\tau X_n}\big] &= \frac{\mu_n}{\big(1 + \tau \theta_n\big)^{k_n+1}}, \\
\mathbb{E}\big[X_n^2 e^{-\tau X_n}\big] &= \frac{\mu_n^2+s_n}{(1+\tau \theta_n)^{k_n+2}}.
\end{aligned}
\label{r1}
\end{align}
where $ k_n=\frac{\mu_n^2}{s_n}, \, \theta_n=\frac{s_n}{\mu_n}$ and $\mu_n, s_n>0, \tau>0$.

The Gamma moment-closure approximation closes the infinite hierarchy of moment equations and allows use to derive consistency conditions to preserve the nonnegativity of the mean of $X_n$ for biological feasibility.

\subsection{Derivation of the Recurrence Relations}

We now derive a system of difference equation for the first and second moments of $X_{n}$ directly from \eqref{eq:model}.

\begin{theorem}[Gamma moment closure equations for the stochastic Ricker map] \label{differeceEq}
Consider the stochastic Ricker difference equation \eqref{eq:model}. Suppose that $\varepsilon_n$ 
is any random variable and independent of $X_n$ with $$E[\varepsilon_n] = 1, \,\,E[\varepsilon_n^2] = v > 1.$$ Assume that $X_n$ satisfies the Gamma moment-closure approximation \eqref{r1}.  Then its mean and variance \((\mu_n,s_n)\) updates in \eqref{eq:model}  can be written as follows:

\begin{equation}\label{eq:closed-mu-s}
(\mu_{n+1},s_{n+1}) = (G_1(\mu_n,s_n), G_2(\mu_n,s_n))
\end{equation}
where
\begin{align}
\begin{aligned}
\mu_{n+1} &= G_1(\mu_n,s_n)= e^{r}\,\frac{\mu_n}{\big(1 + r \theta_n\big)^{k_n+1}}, \\
s_{n+1} & =G_2(\mu_n,s_n)= v\,e^{2r}\,\frac{\mu_n^2+s_n}{(1+2r \theta_n)^{k_n+2}}
- e^{2r}\,\frac{\mu_n^2}{(1+r \theta_n)^{2(k_n+1)}}.
\end{aligned}
\label{expressions}
\end{align}
where $ k_n=\frac{\mu_n^2}{s_n}, \, \theta_n=\frac{s_n}{\mu_n}$.
\end{theorem}

\begin{proof}
Based on the Gaussian moment-closure approximation \eqref{r1}, we obtain
\begin{align}
\mu_{n+1}
&= e^{r}\,\mathbb{E}\big[X_n e^{-r X_n}\big] E[\varepsilon_n]\\
&= e^{r}\,\frac{\mu_n}{\big(1 + r \theta_n\big)^{k_n+1}},
\label{eq:ricker-mu}
\end{align}
and
\begin{align}
\mathbb{E}[X_{n+1}^2]
&= v\,e^{2r}\,\mathbb{E}\big[X_n^2 e^{-2rX_n}\big] E[\varepsilon_n^2] \\
&= v\,e^{2r}\,\frac{\mu_n^2+s_n}{(1+2r \theta_n)^{k_n+2}}.
\label{eq:ricker-2nd}
\end{align}
Hence the variance update
$
s_{n+1} =\mathbb{E}[X_{n+1}^2] - \mu_{n+1}^2
$
gives the closed map

\end{proof}

\begin{remark}\label{rem:limit} Under the assumption of Theorem \ref{differeceEq}. Assume that $ \lim_{n \to \infty } \mu_n=\mu>0$.  The moment dynamical system \eqref{eq:closed-mu-s} will become the classical Ricker map
 \begin{equation}\label{classicalRicker}
 \mu_{n+1} = \mu_n e^{r(1-\mu_n)}   
 \end{equation}
as the variance $s_n \to 0$ and the noise intensity $v \to 1$. For demonstration of the relationship between \eqref{eq:closed-mu-s} and the classical Ricker map \eqref{classicalRicker}, we assume that $\mu_n$ will be close to constant $\mu$ for large $n$, relative to the limit $k_n \to \infty$, thereby justifying the use of the exponential limit theorem.
\end{remark}

\begin{proof}
We consider the limit as the population variance vanishes ($s_n \to 0$). In view of the assumption, we have
 $$k_n = \frac{\mu_n^2}{s_n} \to \infty$$
and 
$$\theta_n = \frac{s_n}{\mu_n} \to 0$$
From these definitions in Theorem \ref{differeceEq}, we observe the fundamental relationship:
\[
k_n \theta_n = \left(\frac{\mu_n^2}{s_n}\right) \left(\frac{s_n}{\mu_n}\right) = \mu_n
\]
This implies $\theta_n =  \frac{\mu_n}{k_n}$.  
Let $D_n$ be the denominator term in the mean equation. We substitute $\theta_n = \frac{\mu_n}{k_n}$:
\[
D_n = \left(1 + r \theta_n\right)^{k_n+1} = \left(1 + \frac{r \mu_n}{k_n}\right)^{k_n+1}
\]
We separate the exponent to isolate the standard exponential limit:
\[
D_n = \left[ \left(1 + \frac{r \mu_n}{k_n}\right)^{k_n} \right] \cdot \left(1 + \frac{r \mu_n}{k_n}\right)^1
\]
Taking the limit as $k_n \to \infty$ and using the definition $e^x = \lim_{t \to \infty} (1 + \frac{x}{t})^t$, we have:
    \[
    \lim_{k_n \to \infty} \left(1 + \frac{r \mu_n}{k_n}\right)^{k_n} = e^{r \mu_n}
    \]
    \[
    \lim_{k_n \to \infty} \left(1 + \frac{r \mu_n}{k_n}\right) = 1 + 0 = 1
    \]
Thus, the denominator converges to:
\[
\lim_{s_n \to 0} D_n = e^{r \mu_n}
\]
Substituting the limit back into the state equation for $\mu_{n+1}$:
\[
\mu_{n+1} = \mu_n e^r \frac{1}{e^{r \mu_n}} = \mu_n e^{r - r\mu_n}
\]
\[
\mu_{n+1} = \mu_n e^{r(1-\mu_n)}
\]
This confirms that as the variance $s_n$ collapses to zero, the system recovers the deterministic Ricker map.

We must verify that the variance $s_{n+1}$ vanishes in the limiting case where $s_n \to 0$ and $v \to 1$.
The variance evolution is given by the difference of two terms, $s_{n+1} = T_1 - T_2$:
\[
T_{1} = v e^{2r} \frac{\mu_n^2+s_n}{(1+2r \theta_n)^{k_n+2}} 
\]
and
\[
T_2 =  e^{2r} \frac{\mu_n^2}{(1+r \theta_n)^{2(k_n+1)}}
\]
As $s_n \to 0$, we have $\mu_n^2 + s_n \to \mu_n^2$. For the denominator, we use the identity $\theta_n = \mu_n / k_n$:
\[
\lim_{k_n \to \infty} (1 + 2r \theta_n)^{k_n+2} = \lim_{k_n \to \infty} \left[\left(1 + \frac{2r\mu_n}{k_n}\right)^{k_n} \cdot \left(1 + \frac{2r\mu_n}{k_n}\right)^2 \right]
\]
\[
= e^{2r\mu_n} \cdot 1 = e^{2r\mu_n}
\]
Substituting this limit and setting $v=1$:
\[
\lim T_1 = e^{2r} \frac{\mu_n^2}{e^{2r\mu_n}} = \mu_n^2 e^{2r(1-\mu_n)}
\]

We analyze the denominator with the exponent $2(k_n+1)$:
\[
\lim_{k_n \to \infty} (1 + r \theta_n)^{2k_n+2} = \lim_{k_n \to \infty} \left[ \left(1 + \frac{r\mu_n}{k_n}\right)^{k_n} \right]^2
\]
\[
= (e^{r\mu_n})^2 = e^{2r\mu_n}
\]
Substituting this limit:
\[
\lim T_2 = e^{2r} \frac{\mu_n^2}{e^{2r\mu_n}} = \mu_n^2 e^{2r(1-\mu_n)}
\]

Subtracting the two asymptotic terms:
\[
\lim_{s_n \to 0} s_{n+1} = T_1 - T_2 = \mu_n^2 e^{2r(1-\mu_n)} - \mu_n^2 e^{2r(1-\mu_n)} = 0
\]
Thus, the system is consistent: if the variance is zero at step $n$, it remains zero at step $n+1$, preserving the deterministic nature of the Ricker map derived in this section.

\end{proof}

\subsection{Validity of the Moment Closure Approximation}
To assess the validity of the Gamma moment approximation,  we performed a series of Monte--Carlo (MC) and moment--map simulations over an increasing sequence of growth rates and multiplicative noise intensity. For the numerical simulations, we evolved an ensemble of \(N = 100{,}000\) independent realizations of the stochastic Ricker map. To maintain consistency with the assumptions of the Gamma moment closure from the initial time \(t=0\), the initial condition \(X_0\) for each realization was drawn independently from a Gamma distribution with mean \(\mu_0 = 0.5\) and variance \(s_0 = 0.02\). Each trajectory was first evolved for a transient period of 1{,}000 iterations, which was discarded to eliminate the influence of initial conditions, and statistics were then collected over the subsequent 500 iterations.  The system subsequently evolves according to the map \eqref{eq:model}, and the multiplicative noise $\varepsilon_n$ is log-normally distributed with $\mathbb{E}[\varepsilon_n]=1$ and $\mathbb{E}[\varepsilon_n^2]=v$. A transient period was discarded to ensure the system reached its stationary state, after which empirical histograms (gray shaded regions) were computed. These are overlaid with the theoretical probability density functions (red curves) derived by iterating the Gamma moment closure recursions for the mean $\mu_n$ and variance $s_n$ in  \eqref{eq:closed-mu-s}, initialized with the exact same parameters $\mu_0$ and $s_0$.

\begin{figure}[h!]
    \centering
    
    \begin{subfigure}[b]{0.32\textwidth}
        \centering
        \includegraphics[width=\textwidth]{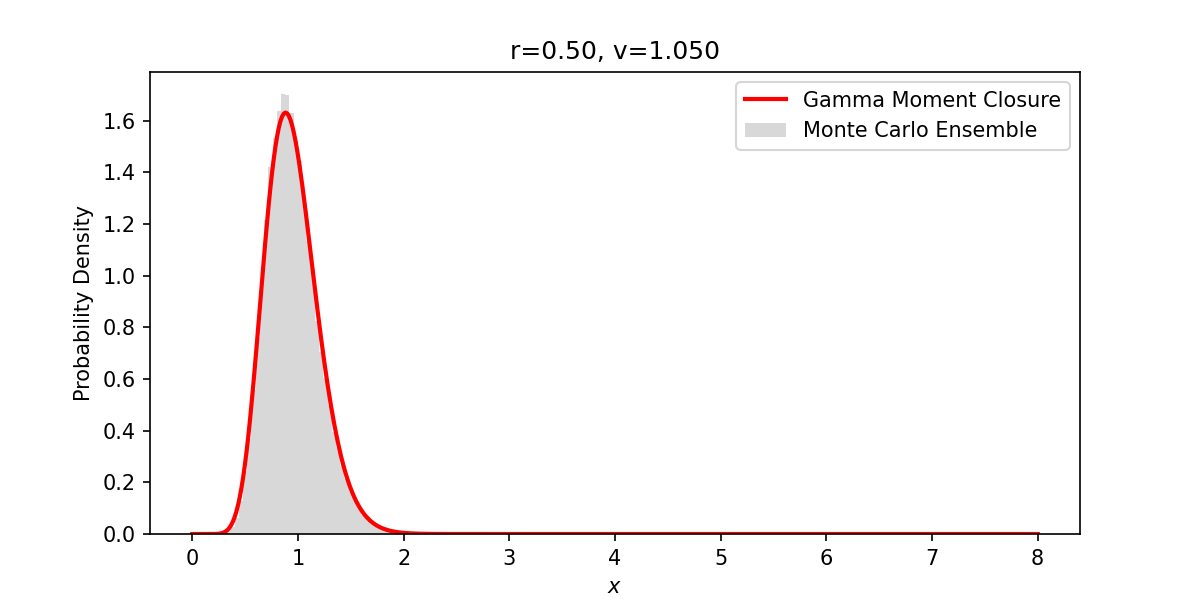}
        \caption{$r=0.50$, $v=1.05$}
        \label{fig:valida_r050_v105}
    \end{subfigure}
    \hfill
    \begin{subfigure}[b]{0.32\textwidth}
        \centering
        \includegraphics[width=\textwidth]{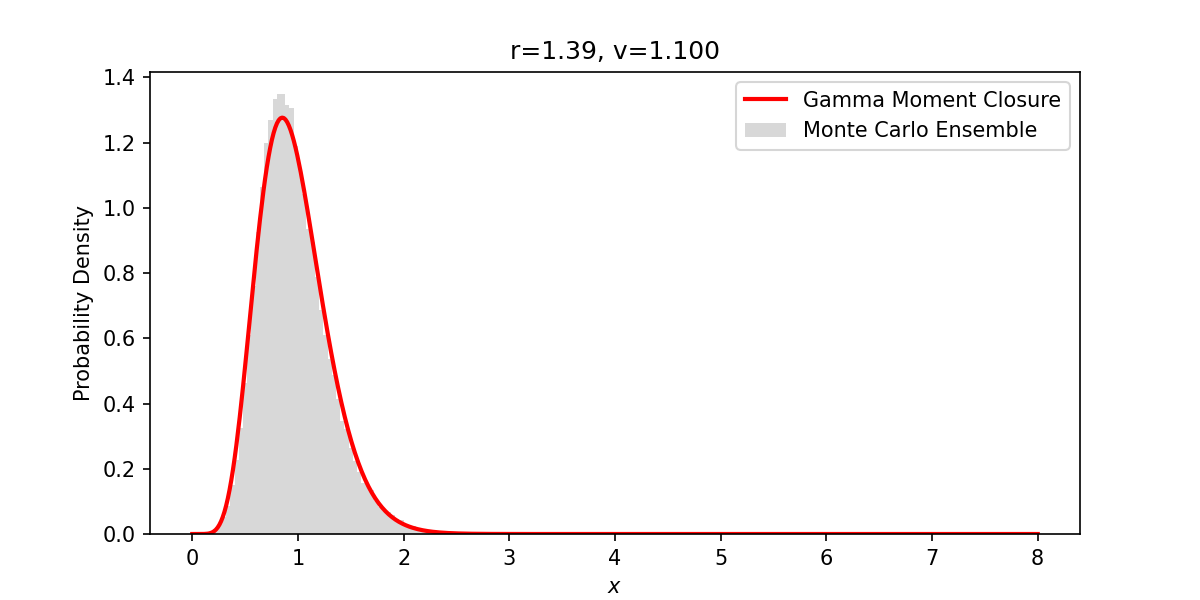}
        \caption{$r=1.39$, $v=1.10$}
        \label{fig:valida_r139_v110}
    \end{subfigure}
    \hfill
    \begin{subfigure}[b]{0.32\textwidth}
        \centering
        \includegraphics[width=\textwidth]{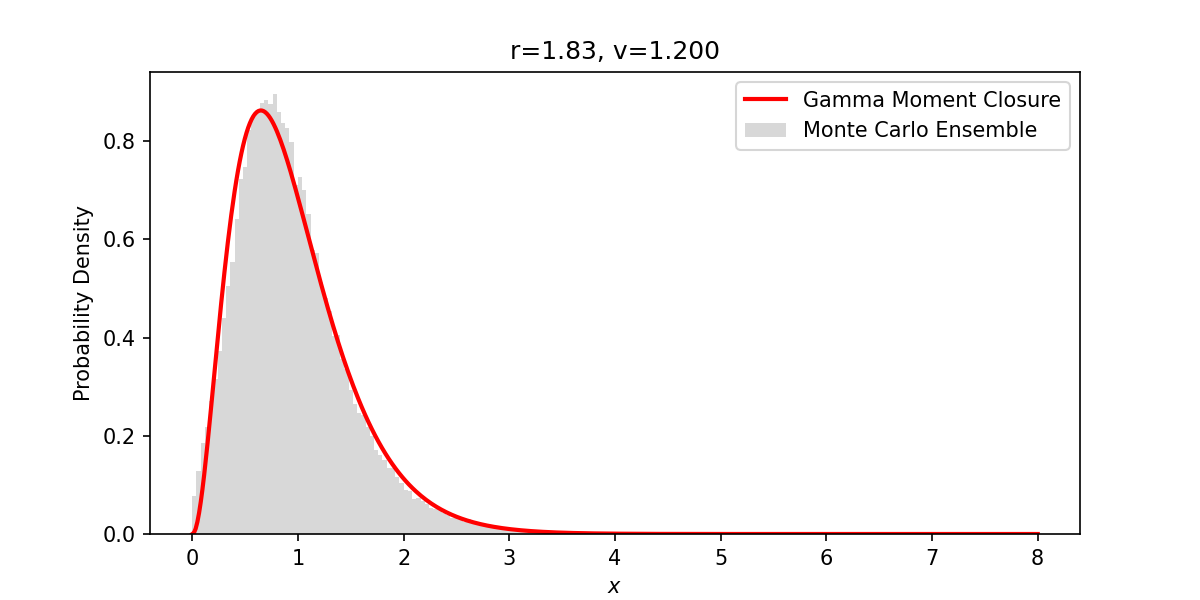}
        \caption{$r=1.83$, $v=1.20$}
        \label{fig:valida_r183_v120}
    \end{subfigure}
    
    \vspace{1em}
    
    \begin{subfigure}[b]{0.32\textwidth}
        \centering
        \includegraphics[width=\textwidth]{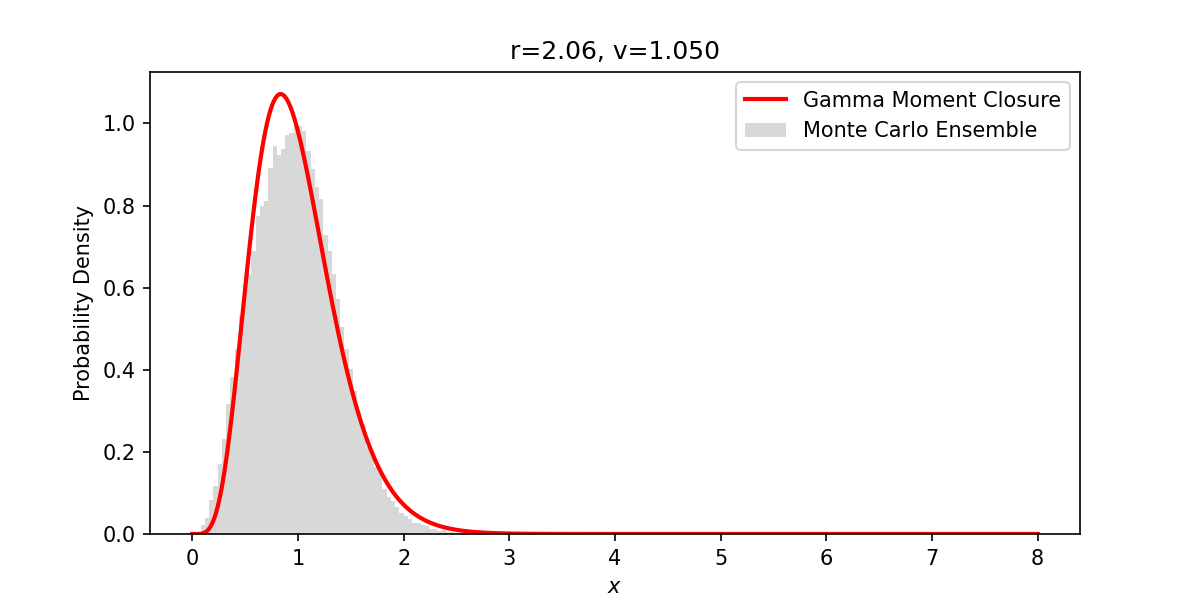}
        \caption{$r=2.06$, $v=1.050$}
        \label{fig:valida_r206_v105}
    \end{subfigure}
    \hfill
    \begin{subfigure}[b]{0.32\textwidth}
        \centering
        \includegraphics[width=\textwidth]{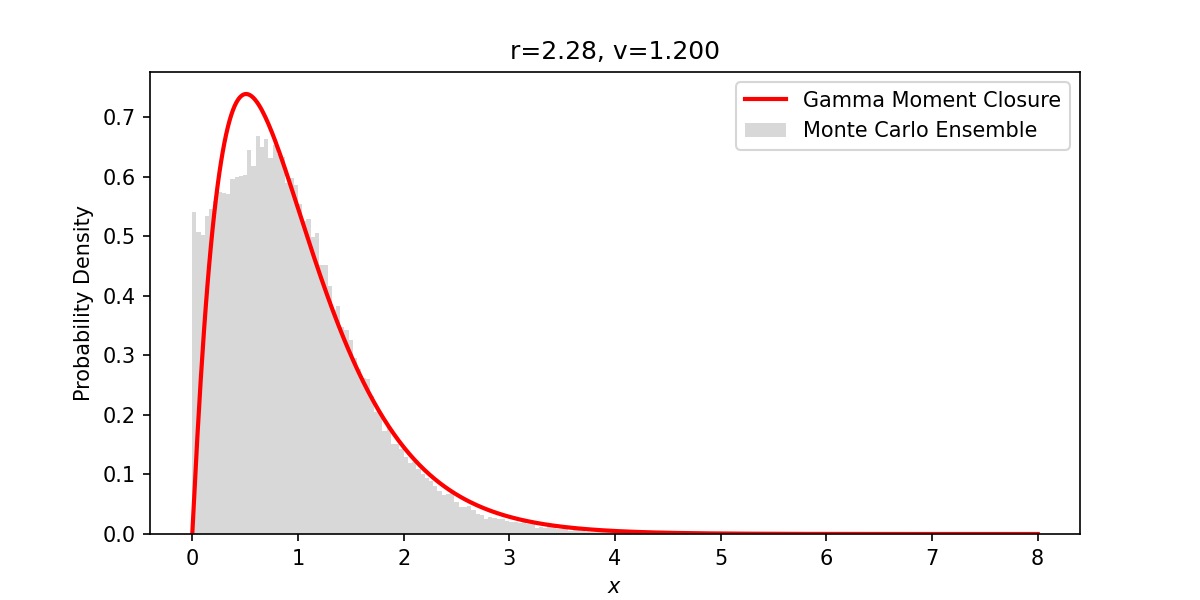}
        \caption{$r=2.28$, $v=1.20$}
        \label{fig:valida_r228_v120}
    \end{subfigure}
    \hfill
    \begin{subfigure}[b]{0.32\textwidth}
        \centering
        \includegraphics[width=\textwidth]{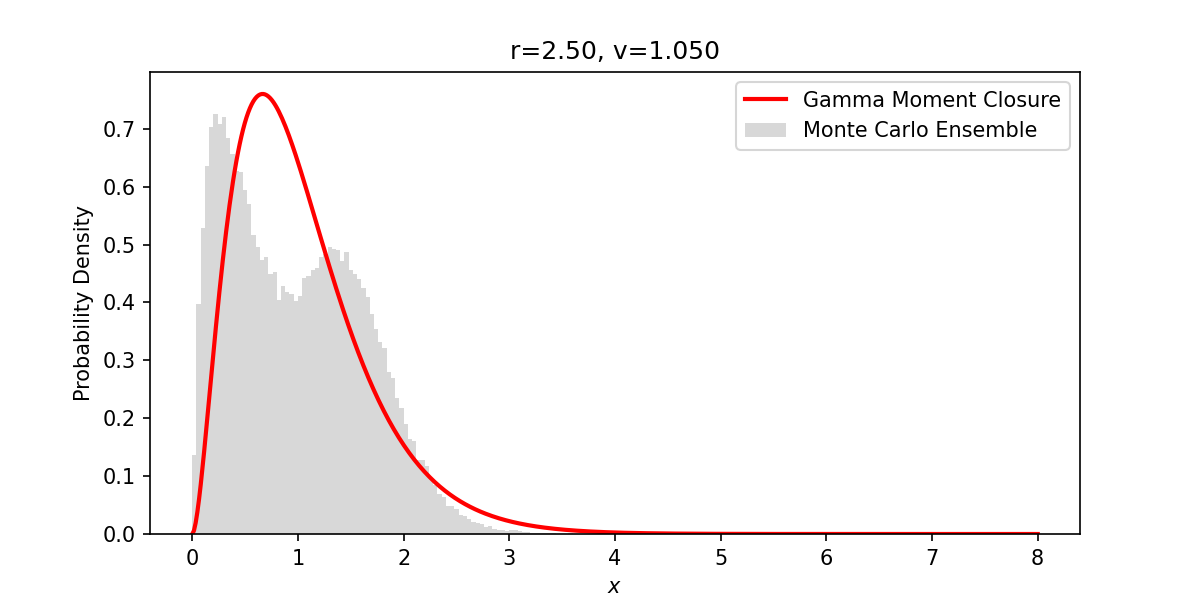}
        \caption{$r=2.50$, $v=1.05$}
        \label{fig:valida_r250_v105}
    \end{subfigure}
    
    \caption{Comparison of stationary distribution with Gamma moment--closure dynamics.}
    \label{fig:valida}
\end{figure}

Figure~\ref{fig:valida} compares the stationary distribution of the stochastic Ricker map obtained using two distinct approaches: direct numerical Monte--Carlo simulations of the stochastic system \eqref{eq:model} and the analytical prediction provided by the Gamma moment closure \eqref{eq:closed-mu-s}. The comparison demonstrates that the accuracy of the Gamma approximation depends sensitively on the underlying parameter regime.

In the regime \(0<r<2\), where the deterministic Ricker map \eqref{classicalRicker} possesses a globally stable equilibrium \cite{may1976simple}, the stationary distribution of \eqref{eq:model} is well approximated by the Gamma moment closure even under relatively strong multiplicative noise (\(1<v \leq 1.2\)). This agreement is evident in Figures~\ref{fig:valida_r050_v105}, \ref{fig:valida_r139_v110}, and \ref{fig:valida_r183_v120}, where the empirical Monte--Carlo histograms closely match the corresponding Gamma distributions in both shape and spread.

When \(r>2\), the deterministic Ricker map \eqref{classicalRicker} becomes unstable and undergoes a sequence of period-doubling bifurcations \cite{may1976simple}. In this regime, the performance of the Gamma moment closure becomes more sensitive. For moderate noise intensities (\(1<v \leq 1.20\)), the stationary distribution is still captured reasonably well by the Gamma approximation, as illustrated in Figures~\ref{fig:valida_r206_v105} and ~\ref{fig:valida_r228_v120}. However, as $r$ increases,  even with small noise intensity (\(v =1.05\)), significant discrepancies emerge. In Figure \ref{fig:valida_r250_v105}, the Gamma moment closure fails to reproduce the noticeable peaks of the empirical stationary distributions. Nevertheless, it still captures the overall tail behavior reasonably well, providing a qualitatively accurate description of the decay of the distribution at large population sizes.

These observations indicate that while the Gamma moment closure provides an effective low-dimensional description in deterministically stable regimes, its validity deteriorates in parameter regions where nonlinear instabilities and strong noise jointly shape the stationary distribution.

\section{Feasible Positive Equilibrium}\label{feasb}

Note that both $\mu_n$ and $s_n$ in system \eqref{eq:closed-mu-s} can not be zero, our attention turns to feasible positive equilibria $(\mu,s)$ such that $\mu>0,s>0$. We begin by constructing an auxiliary function to analyze the existence and uniqueness of the positive equilibrium.

\subsection{Auxiliary function $F$}

\begin{theorem}[ Auxiliary function $F$]\label{auxi} Under the assumptions of Theorem \ref{differeceEq}. Assume that for $r>0, v>1$, and  define the following function of $z$ for \(z>0\): 
\begin{equation}\label{F1}
F(z;r,v)\;:=\;\Big(\frac{r}{\ln(1+z)}+1\Big)\ln(1+2z)\;-\;2r-\ln v,
\qquad z>0. 
\end{equation}
Then \((\mu,s)\) with $\mu>0,s>0$ is a feasible equilibrium  of Gamma moment-closure system \eqref{eq:closed-mu-s} if and only if
\begin{equation}\label{F2}
\mu \;=\; \frac{z\big(r-\ln(1+z)\big)}{r\ln(1+z)},
\qquad
s \;=\; \frac{\mu z}{r}.
\end{equation}
where $z>0$ is the solution of $F(z;r,v)=0$.
\end{theorem}

\begin{proof}
Assume $(\mu, s)$ with \(\mu>0,s>0\) is an equilibrium of \eqref{eq:closed-mu-s}. Note that $k=\frac{\mu^2}{s}, \theta=\frac{s}{\mu}$. From the first equilibrium condition of \eqref{eq:closed-mu-s} 
\[
\mu = e^{r}\,\frac{\mu}{(1 + r\theta)^{k+1}}
\]
we may cancel \(\mu\) and obtain
\[
(1 + r\theta)^{k+1} = e^{r}.
\]
Taking logarithms yields the identity
\begin{equation}\label{ad}
(k+1)\ln(1+r\theta) = r. 
\end{equation}
Introduce the auxiliary variable
\[
z := r\theta = \frac{r s}{\mu},
\]
so that \(1+r\theta=1+z\) and \(1+2r\theta=1+2z\).  Using \(k=\dfrac{\mu^2}{s}\) and \(\theta=\dfrac{s}{\mu}\) one finds
\[
k+1 = \mu \frac{\mu}{s}+1 = \mu \frac{r}{z}+1= \frac{\mu r}{z}+1.
\]
Thus equation \eqref{ad} becomes
\[
\Big(\frac{\mu r}{z}+1\Big)\ln(1+z) = r.
\]
Solving this for \(\mu r/z\) gives
\[
\frac{\mu r}{z} = \frac{r}{\ln(1+z)} - 1.
\]
Hence
\[
\mu \;=\; \frac{z}{r}\Big(\frac{r}{\ln(1+z)} - 1\Big)
= \frac{z\big(r-\ln(1+z)\big)}{r\ln(1+z)},
\]
which is the first relation in \eqref{F2}. Then \(s=\mu\theta=\mu z/r\), producing the second relation in  \eqref{F2}.

Next we derive \eqref{F1}.  From the first equilibrium condition of \eqref{eq:closed-mu-s}, we have
\[
(1 + r\theta)^{2(k+1)} = e^{2r},
\]
so the last term in the variance update simplifies at equilibrium:
\[
e^{2r}\,\frac{\mu^2}{(1 + r\theta)^{2(k+1)}} = \mu^2.
\]
Thus the variance equilibrium relation reduces to
\[
s = v e^{2r}\,\frac{\mu^2+s}{(1+2r\theta)^{k+2}} - \mu^2,
\]
or equivalently
\[
(\mu^2+s)\Big[1 - \frac{v e^{2r}}{(1+2r\theta)^{k+2}}\Big] = 0.
\]
Since \(\mu^2+s>0\) for any nontrivial equilibrium, we obtain
\begin{equation}\label{afd}
(1+2r\theta)^{k+2} = v e^{2r}.
\end{equation}
Since  $$1+2r\theta=1+2z$$  and by \eqref{ad}, it follows  
$$k+2 = (k+1)+1 = \dfrac{r}{\ln(1+z)}+1$$
Thus, taking logarithms and substituting in \eqref{afd}, we have 
\[
\Big(\frac{r}{\ln(1+z)}+1\Big)\ln(1+2z) = 2r + \ln v,
\]
which rearranges to the scalar equation \(F(z;r,v)=0\) displayed in \eqref{F1}. Now note that the derivation works for both the ``if" direction and the ``only if" direction. This completes the reduction: any positive equilibrium of  \eqref{eq:closed-mu-s} corresponds to a positive root \(z>0\) of \(F\) in  \eqref{F1}, and \((\mu,s)\) is then given by \eqref{F2}.
\end{proof}

\subsection{Existence of feasible equilibrium}
\begin{theorem}[Existence of Feasible Gamma closure equilibrium]\label{existence} Under the assumptions of Theorem \ref{differeceEq} and assume that
\begin{equation}\label{condition}
v \;<\; R(r):=\big(2-e^{-r}\big)^2.
\end{equation}
Let
\[
F(z;r,v)\;=\;\Big(\frac{r}{\ln(1+z)}+1\Big)\ln(1+2z)\;-\;2r-\ln v,\qquad z>0,
\]
with parameters $r>0$ and $v>1$. Then for every $r>0$ and $v>1$, the scalar equation
\[
F(z;r,v)=0,\qquad z>0,
\]
have at least one solution $z^*>0$ with $r-\ln(1+z^*)>0$. Consequently, Gamma moment-closure system \eqref{eq:closed-mu-s} possesses a nontrivial equilibrium $(\mu^*,s^*)$ (with $\mu^*>0,\ s^*>0$), given by
\[
\mu^*=\frac{z^*\big(r-\ln(1+z^*)\big)}{r\ln(1+z^*)},\qquad
s^*=\frac{\mu^* z^*}{r}.
\]
\end{theorem}

\begin{remark}
As we will see in Theorem \ref{rem1}, $v < R(r) := (2 - e^{-r})^2$  is both sufficient and necessary condition for the existence of feasible solution of \eqref{eq:closed-mu-s}.  The condition $v < R(r) := (2 - e^{-r})^2$, as shown in Figure \ref{fig:r1},  characterizes a critical boundary for the existence and stability of the positive equilibrium. Biologically, this inequality reflects the necessary balance between the stabilizing force of the intrinsic growth rate $r$ and the destabilizing potential of environmental variability $v$. Since the threshold function $R(r)$ is monotonically increasing, a higher growth rate enhances the population's resilience, allowing it to absorb larger stochastic fluctuations while maintaining a stable state. Crucially, the upper bound $\lim_{r \to \infty} R(r) = 4$ implies a fundamental theoretical limit: if the noise intensity reaches $v \geq 4$, the environmental perturbations become excessive, rendering the positive equilibrium unstable regardless of the magnitude of the population's intrinsic growth rate.
\end{remark}

\begin{figure}[h!]
           \centering
        \includegraphics[width=0.8\textwidth]{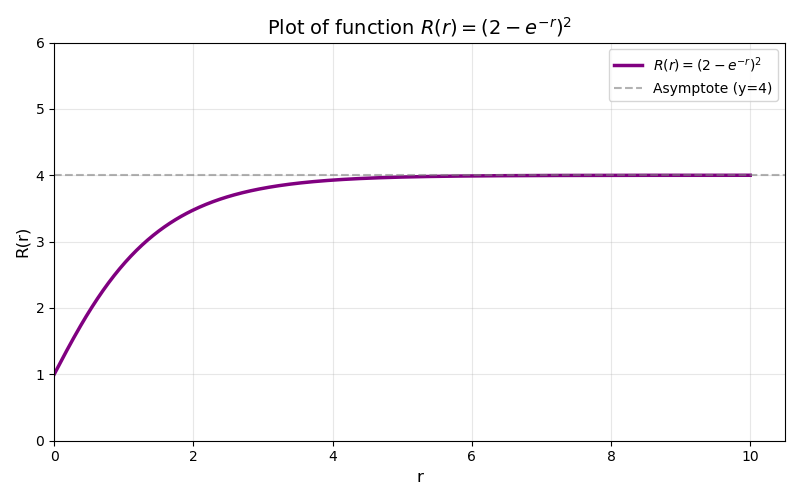}   
    \caption{Graphs of $R(r) = (2 - e^{-r})^2$ }
     \label{fig:r1}
\end{figure}

\begin{figure}[h!]
    \centering
    
    \begin{subfigure}[b]{0.45\textwidth}
        \centering
        \includegraphics[width=\textwidth]{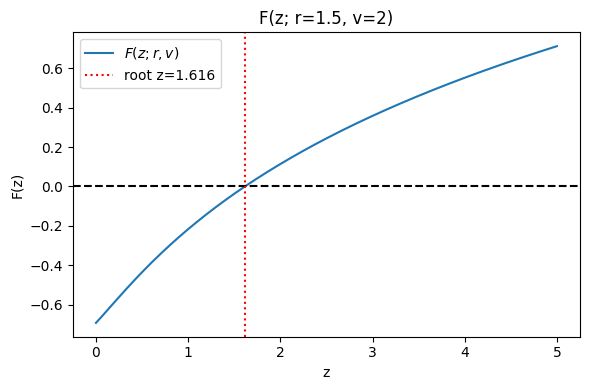}
        \caption{$r=1.5$, $v=2$, $z=1.616$ }
    \end{subfigure}
    \hfill
    \begin{subfigure}[b]{0.45\textwidth}
        \centering
        \includegraphics[width=\textwidth]{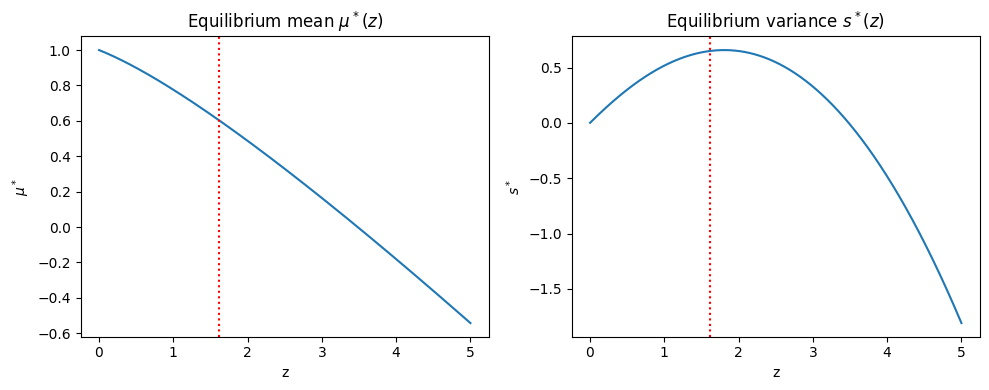}
        \caption{$r=1.5$, $v=2$, $\mu^*=0.602$,$s^*=0.649$ }
   \end{subfigure}
   
   \vspace{1em}  

    \begin{subfigure}[b]{0.45\textwidth}
        \centering
        \includegraphics[width=\textwidth]{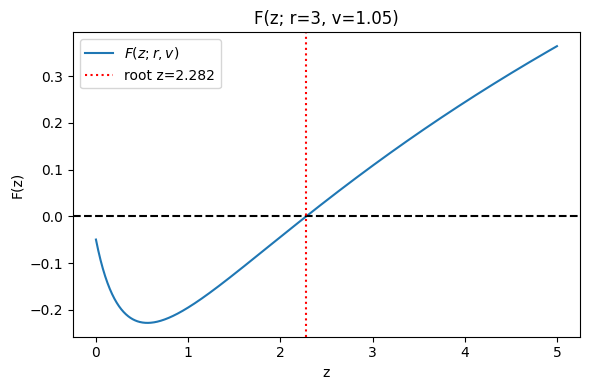}
        \caption{$r=3$, $v=1.05$, $z=2.282$ }
    \end{subfigure}
    \hfill
    \begin{subfigure}[b]{0.45\textwidth}
        \centering
        \includegraphics[width=\textwidth]{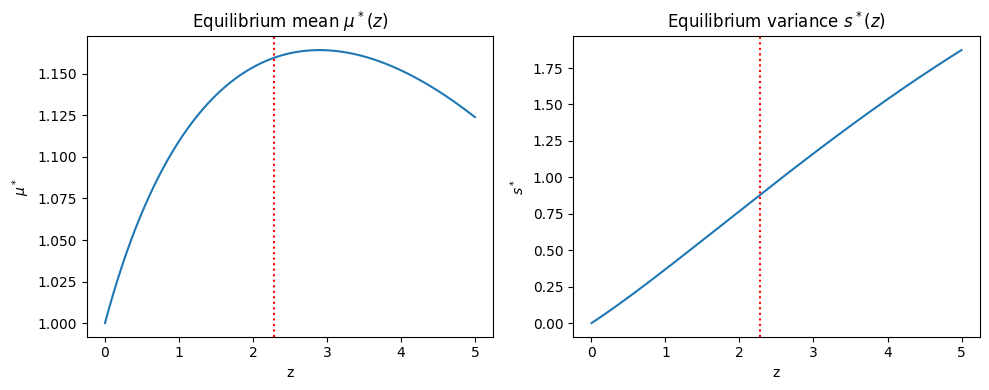}
        \caption{$r=3$, $v=1.05$, $\mu^*=1.159$,$s^*=0.88$ }
    \end{subfigure}

  \caption{Graph of $F(\cdot;r,v)$ and $(\mu^*,s^*)$}
     \label{fig:func1}
\end{figure}
\begin{proof}
First, the function $F$ (we often use $F(z)$ instead of $F(z;r,v)$ for simplicity) is continuous on $(0,\infty)$. As $z >0$ approaches $0$, 
\[
\frac{\ln(1+2z)}{\ln(1+z)}\to 2.
\]
Hence
\[
\lim_{z\downarrow0}\!\Big(\frac{r}{\ln(1+z)}+1\Big)\!\ln(1+2z)
= \lim_{z\downarrow0}\!\frac{r\,\ln(1+2z)}{\ln(1+z)} + \ln(1+2z)
=2r.
\]
Therefore
\[
\lim_{z\downarrow0}F(z;r,v)=-\ln v.\;
\]
Because $v>1$, we have $F(0^+)=-\ln v<0$.

Next, as $z\to\infty$, note that $\frac{r}{\ln(1+z)}\to 0$. Thus
\[
F(z;r,v)\sim \ln(1+2z)-2r-\ln v\to+\infty,\qquad z\to\infty.
\]
By continuity, since $F(0^+)<0$ and $F(z)\to+\infty$, there exists at least one $z^*>0$ such that $F(z^*;r,v)=0$. Hence a positive root exists.

Now we want to show $z^*>0$  results in a feasible solution. Set
\[
z_{\min}:=e^{r}-1>0,
\]
so that \(\ln(1+z_{\min})=r\). Then
\[
\frac{r}{\ln(1+z_{\min})}=1,
\]
and hence the definition of \(F\) yields the exact identity
\begin{equation}
F(z_{\min};r,v)
=2\ln\big(1+2z_{\min}\big)-2r-\ln v
=2\ln\!\big(2e^{r}-1\big)-2r-\ln v.
\label{eq:F_at_zmin}
\end{equation}
Thus
\begin{equation}
F(z_{\min};r,v)>0
\quad\Longleftrightarrow\quad
\ln v \;<\; 2\ln\!\big(2e^{r}-1\big) - 2r
\quad\Longleftrightarrow\quad
v \;<\; \frac{(2e^{r}-1)^2}{e^{2r}} \;=\; (2-e^{-r})^2.
\label{eq:vcondition}
\end{equation}
By assumption \(v<(2-e^{-r})^2\), so \eqref{eq:F_at_zmin} gives 
$$F(z_{\min};r,v)>0$$
By continuity of \(F\) there exists a root in the interval \((0,z_{\min})\). Hence we can choose a solution of $F=0$, $z^*>0$, so that 
$$
0<z^*<z_{\min}=e^{r}-1
$$
therefore,  $$r>\ln(1+z^*)$$ 

Since $z^*>0$ and $r>\ln(1+z^*)$, both $\mu^*>0$ and $s^*>0$. Thus, by Theorem \ref{auxi}, for every $r>0$ and $v>1$, the function $F(\cdot;r,v)$ has at least one positive root $z^*>0$, yielding a positive Gamma--closure equilibrium $(\mu^*,s^*)$: 
\[
\mu^*=\frac{z^*\big(r-\ln(1+z^*)\big)}{r\ln(1+z^*)},\qquad
s^*=\frac{\mu^* z^*}{r}.
\]
Figure \ref{fig:func1} is the graph of $F(\cdot;r,v)$ with one positive root $z^*>0$, yielding a positive Gamma--closure equilibrium $(\mu^*,s^*)$.
\end{proof}

\subsection{Uniqueness of feasible equilibrium}
\begin{theorem}[Uniqueness of Feasible Gamma closure equilibrium]\label{uniqueness} Under the assumptions of Theorem \ref{differeceEq}, 
let
\[
F(z;r,v)\;=\;\Big(\frac{r}{\ln(1+z)}+1\Big)\ln(1+2z)\;-\;2r-\ln v,\qquad z>0,
\]
with parameters $r>0$ and $v>1$. then for every $r>0$ and $v>1$, the scalar equation
\[
F(z;r,v)=0,\qquad z>0,
\]
has only one solution $z^*>0$. Consequently, Gamma moment-closure system \eqref{eq:closed-mu-s} possesses at most one feasible equilibrium.
\end{theorem}

\begin{proof} Note that $v \;<\; \big(2-e^{-r}\big)^2$ is not assumed in Theorem \ref{uniqueness} for uniqueness. Lemma \ref{Phi} in Appendix is essential for the proof of Theorem \ref{uniqueness}. The proof of Lemma \ref{Phi} is lengthy and organized in Appendix. Let the function $\Phi(z)$ be defined for $z>0$ as:
\[
\Phi(z):=\frac{\ln(1+2z)}{\ln(1+z)}.
\]
so that
\[
F(z;r,v)=r\big(\Phi(z)-2\big)+\ln(1+2z)-\ln v.
\]
Differentiating $F$ (we often use $F(z)$ instead of $F(z;r,v)$ for simplicity) with respect to $z$ yields the first and second derivatives:
\[
F'(z)=r\,\Phi'(z)+\frac{2}{1+2z}, \qquad F''(z)=r\,\Phi''(z)-\frac{4}{(1+2z)^2}.
\]
Now set $$H(z) \equiv \Phi''(z)(1+2z)^2$$
we can factor $F''(z)$ as follows:
\[
F''(z) = \frac{r}{(1+2z)^2} \left( H(z) - \frac{4}{r} \right).
\]
Since the pre-factor $\frac{r}{(1+2z)^2}$ is strictly positive for all $z$, the sign of $F''(z)$ is determined solely by the sign of the term $(H(z) - 4/r)$.

\begin{figure}[h!]
           \centering
        \includegraphics[width=0.8\textwidth]{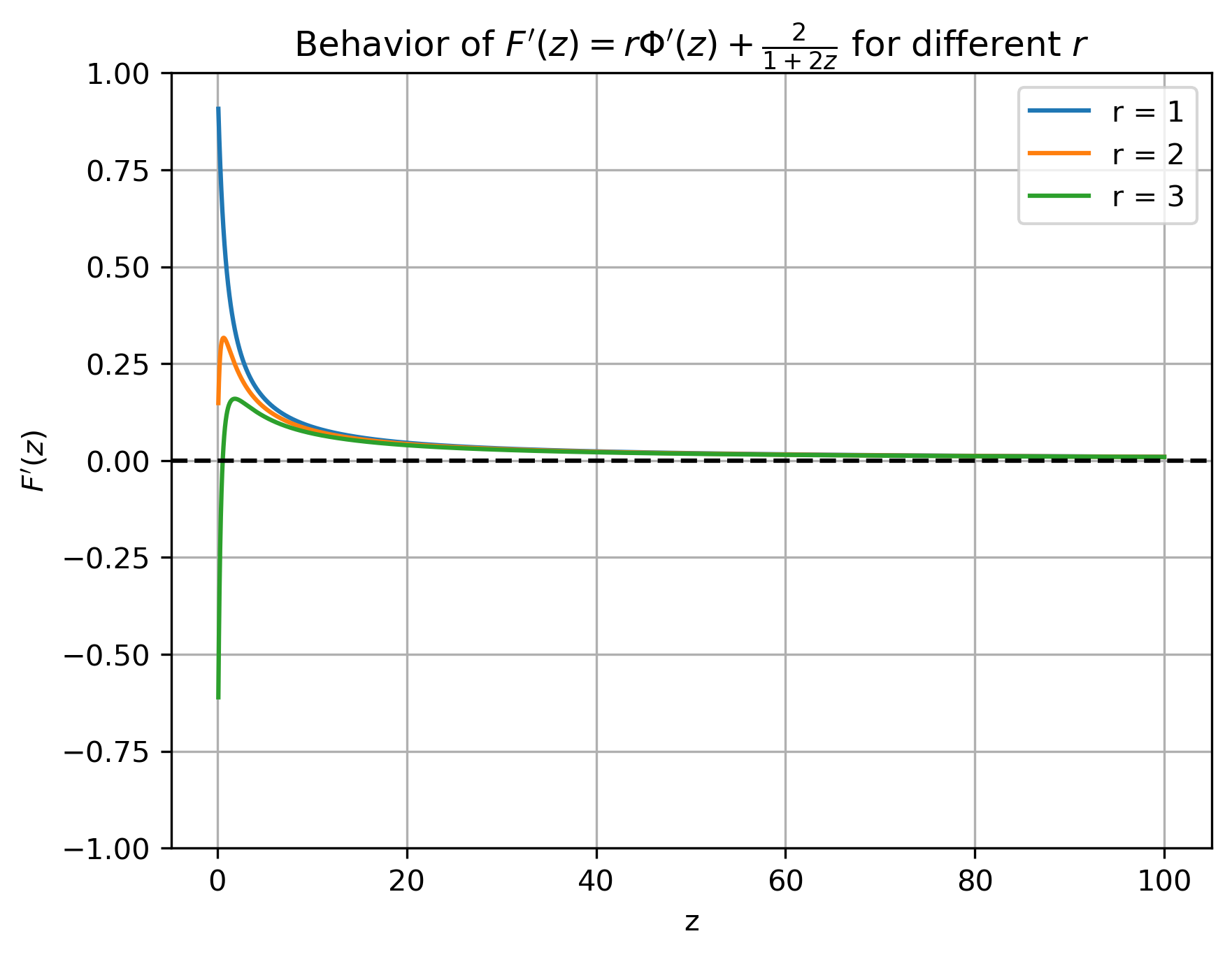}   
    \caption{Graphs of $F'(z)=r\Phi'(z)+\frac{2}{1+2z}$}
     \label{fig:F1122}
\end{figure}
In view of Lemma \ref{Phi} and as depicted in Figures  \ref{fig:phi} and \ref{fig:H}, $H(z)$ is a strictly decreasing function on $(0, \infty)$ starting from $H(0)=3$ and approaching 0 as $z \to \infty$. This monotonicity dictates the sign of $F''(z)$ through two distinct cases:
\begin{itemize}
    \item Case $r \le 4/3$ (Concave down): Since $H(z) < 3 \le 4/r$, the term $(H(z) - 4/r)$ is always negative. Thus, $F''(z) < 0$ for all $z$.
    \item Case $r > 4/3$ (Convex then Concave down): Since $H(z)$ starts at 3 ($>4/r$) and decreases to 0, it intersects the value $4/r$ exactly once at a point $z_{\text{infl}}$. Consequently, $F''(z) > 0$ for $z < z_{\text{infl}}$ (convex) and $F''(z) < 0$ for $z > z_{\text{infl}}$ (concave down).
\end{itemize}

Next, we use this concavity behavior to determine the monotonicity of $F(z)$ by analyzing $F'(z)$ in the two following two cases. First, Figure \ref{fig:F1122} demonstrates that the graph of $F'$. From Lemma \ref{Phi}, it follows that 
\[
F'(0^+)=2-r, \,\,  \lim_{z \to \infty} F'(z) = 0.
\]
\textbf{Case 1: $0 < r \le 2$}. \\
In this case, $F'(0^+) = 2 - r \ge 0$. If $r \le 4/3$, $F(z)$ is strictly concave down ($F''<0$), implying $F'(z)$ strictly positive for all $z>0$ because $lim_{z \to \infty} F'(z) = 0.$  If $4/3 < r \le 2$, because of $F''(z) > 0$ for $z < z_{\text{infl}}$, $F'(z)$ is strictly increasing on $(0, z_{\text{infl}}]$ starting from a non-negative value $F'(0^+) \ge 0$. It follows that $F'(z)$ is strictly positive on $(0, z_{\text{infl}}]$. On the interval $(z_{\text{infl}}, \infty)$, $F'(z)$ is strictly increasing too. Suppose, for the sake of contradiction, that $F'(z)$ becomes non-positive at some point $z_0 > z_{\text{infl}}$ (i.e., $F'(z_0) \le 0$). Since $F''(z)<0$ in this region, for all $z > z_0$, we would have $F'(z) < F'(z_0) \le 0$. This would imply:
\[
\lim_{z \to \infty} F'(z) \le F'(z_0) < 0.
\]
However, this contradicts the established limit $\lim_{z \to \infty} F'(z) = 0$.

Combining this with the positive increasing phase, we conclude that $F'(z) > 0$ for all $z > 0$. Consequently, $F(z)$ is strictly increasing on $(0, \infty)$.

\textbf{Case 2: $r > 2$.} \\
The initial slope is negative ($F'(0^+) < 0$) and $F(z)$ must eventually grow to $+\infty$. Therefore, the derivative $F'(z)$ cannot remain negative; it must cross zero to become positive. Assume that $\hat{z}>0$ is the smallest $z$ such that $F'(\hat{z})=0.$  We claim that $F'(z)$ remains strictly positive for all $z > \hat{z}$. Suppose, for the sake of contradiction, that this is not true. This would imply there exists some $z_1 > \hat{z}$ such that $F'(z_1) < 0$. This would imply that $F''(z) = 0$ has at least \textit{two} distinct solutions because $F(z)$ must eventually grow to $+\infty$. However, we have proven that $H(z)=0$, ( therefore $F''(z) = 0$), has at most one solution. This contradiction implies that $F'(z)$ cannot cross back to negative values. Thus, $F'(z) > 0$ for all $z > \hat{z}$.

Finally, we prove the uniqueness of the solution $z^*$.
For $r \le 2$, $F(z)$ is strictly increasing from $F(0^+)$ to $+\infty$ and the function must cross zero exactly once.
For $r > 2$, $F(z)$ decreases on $(0, \hat{z}]$ to a negative minimum $F(\hat{z}) < F(0^+)$, so there are no solutions in this interval. On the interval $(\hat{z}, \infty)$, $F(z)$ is strictly increasing from a negative value to $+\infty$. Thus, it intersects the $z$-axis exactly once at a point $z^* > \hat{z}$.

In all cases, the equation $F(z;r,v)=0$ possesses exactly one positive solution $z^*$.
\end{proof}

\subsection{Sufficient and necessary condition for existence}

\begin{theorem}\label{rem1} Under the assumptions of Theorem \ref{differeceEq}, let
\[
F(z;r,v)\;=\;\Big(\frac{r}{\ln(1+z)}+1\Big)\ln(1+2z)\;-\;2r-\ln v,\qquad z>0,
\]
with parameters $r>0$ and $v>1$.  The condition $$v \;<\; \big(2-e^{-r}\big)^2$$ is both sufficient and necessary for the existence of feasible solution of \eqref{eq:closed-mu-s}.
\end{theorem}

\begin{proof}
Note that $v \;<\; \big(2-e^{-r}\big)^2$ is not assumed in Theorem \ref{uniqueness} for uniqueness. The sufficient part of Theorem \ref{rem1} is the result of Theorems \ref{existence}. We now need to prove that if $v \; \geq \; \big(2-e^{-r}\big)^2$, the unique positive solution $z_0>0$ of $F(z;r,v)=0$ does not satisfy the feasible condition 
$$
r>\ln(1+z_0)
$$   
In the proof of Theorem \ref{existence}, in view of \eqref{eq:vcondition},  
if $v \; \geq \; \big(2-e^{-r}\big)^2$, then 
$$ 
F(z_{\min};r,v) \leq 0
$$
where  $z_{\min}=e^r-1$. We will proceed the proof in the two cases: \\

\textbf{Case 1: $0 < r \le 2$}. As in the proof of Theorem \ref{uniqueness}, $F'(z) > 0$ for all $z > 0$. Consequently, $F(z)$ is strictly increasing on $(0, \infty)$. Because $F(z_{\min}) \leq  0=F(z_0)$, it follows that 
$$
e^r-1=z_{\min} \leq  z_0. 
$$
which is 
$$
r \leq \ln(1+z_0)
$$
As a result, $z_0$ is not a feasible equilibrium. \\
\textbf{Case 2: $r > 2$.} As in the proof of Theorem \ref{uniqueness}, assume that $\hat{z}>0$ is the smallest $z$ such that $F'(\hat{z})=0.$ We prove that $F(z)$ decreases on $(0, \hat{z}]$ and strictly increases on the interval $(\hat{z}, \infty)$. If we can show that
$$
\hat{z} < z_{min}
$$
because of the monotonicity of $F$ on $(\hat{z}, \infty)$,  as in Case 1, we have $e^r-1=z_{\min} \leq z_0$ and $r \leq \ln(1+z_0)$ and $z_0$ is not a feasible equilibrium.

Now we only need to verify that $F'(z_{min})>0$, which implies that $\hat{z} < z_{min}$ because the definition of $\hat{z}$. Recall in Theorem \ref{uniqueness}
\begin{equation}
    F'(z) = r \Phi'(z) + \frac{2}{1+2z}
\end{equation}
where $\Phi'(z) = \frac{u'v - uv'}{v^2}$, with $u = \ln(1+2z)$ and $v = \ln(1+z)$. We claim that $F'(z_{min})$ is strictly positive for all $r > 0$. Let $z=z_{min}= e^r - 1$,  therefore $r = \ln(1+z)$ and evaluate $F'(z)$ at $z=z_{min} = e^r - 1$,
\begin{align}
    F'(z) &= r \left[ \frac{r u' - u v'}{r^2} \right] + \frac{2}{1+2z} \\
          &= \frac{r u' - u v'}{r} + \frac{2}{1+2z} \\
          &= u' - \frac{u v'}{r} + \frac{2}{1+2z}
\end{align}
Substituting the derivatives $u' = \frac{2}{1+2z}$ and $v' = \frac{1}{1+z}$, again noting $r = \ln(1+z)$:
\begin{align}
    F'(z) &= \frac{2}{1+2z} - \frac{\ln(1+2z)}{r(1+z)} + \frac{2}{1+2z} \\
          &= \frac{4}{1+2z} - \frac{\ln(1+2z)}{(1+z)\ln(1+z)}\\
          &= \frac{1}{1+z} [\frac{4(1+z)}{1+2z} - \frac{\ln(1+2z)}{\ln(1+z)}]\\
          &= \frac{1}{1+z} [\frac{2}{1+2z} + (2 - \frac{\ln(1+2z)}{\ln(1+z)})]          
\end{align}
Since $(1+z)^2 > 1 + 2z$ for all $z>0$, taking the natural logarithm of both sides:
\begin{equation}
 2\ln(1+z) > \ln(1+2z)
\end{equation}
It follows that  $F'(z_{min})=F'(e^r - 1)=F'(z)>0$.  In summary, we conclude that if $v \; \geq \; \big(2-e^{-r}\big)^2$, the unique positive solution $z_0>0$ of $F(z;r,v)=0$ does not satisfy the feasible condition $r >\ln(1+z_0)$.  
\end{proof}
We employed numerical methods to verify the existence and uniqueness of the feasible solution for $v < R(r)=(2 - e^{-r})^2$ and to confirm the non-existence of feasible solutions when $v > (2 - e^{-r})^2$. We performed a high-resolution parameter scan over $r \in [0.5, 10]$ and $v \in [1.05, 4.5]$, computing roots of the auxiliary function $F(z)$ via a hybrid grid-search and Brent's method 
\cite{brent1973algorithms}. As shown in Figure \ref{map_1_existence}, the results numerically confirm that solutions are unique within the existence region (blue), as no multiple roots were detected. Furthermore, the solver found no feasible solutions in the region where $v > (2 - e^{-r})^2$ (gray), numerically validating the theoretical boundary for extinction. The simulation is performed with a search limit of $z_{\text{max}} = 30,000$. Since we restrict $r$ to the interval $[0, 10]$, from the proof Theorem \ref{rem1} that, any $z$ larger than $ e^r - 1$ can not be a feasible solution, which ensures that the solution never exceeds $e^{10} - 1 \approx 22,025$.  Because the search space is strictly larger than this theoretical upper bound, the simulation is valid and guaranteed to capture the solutions for the entire range of $r$. We numerically verified that the analytic curve \(v=R(r)\) separates the parameter region: 
\[
\begin{cases}
v < R(r) &\text{the recovered equilibrium satisfies }\mu^*>0,\ s^*>0,\\[4pt]
v > R(r) &\text{the recovered equilibrium satisfies } \mu^*\le 0\ \text{or}\ s^*\le 0.
\end{cases}
\]

\begin{figure}[h!]
    \centering
        \includegraphics[width=0.8\textwidth]{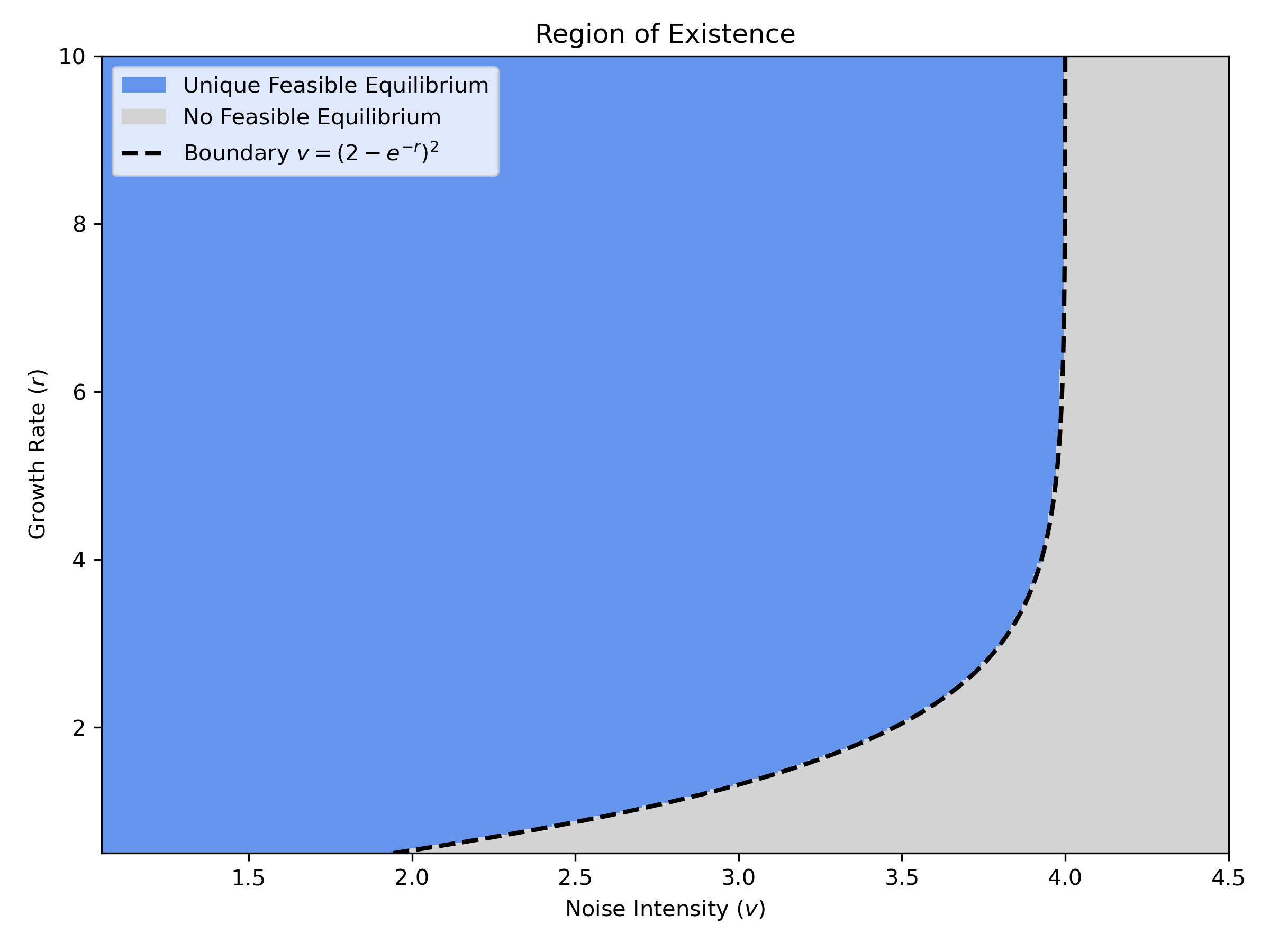}
        \caption{Existence and uniqueness region}
        \label{map_1_existence}
\end{figure}

\section{Stability Region of Equilibria}\label{stables}

\subsection{Equivalence Between Feasibility and Stability}

In view of Theorem \ref{rem1}, the feasible equilibrium of \eqref{eq:closed-mu-s} exists and is unique.  The local stability of each equilibrium is determined by the eigenvalues of the Jacobian matrix 
evaluated at that equilibrium point. If all eigenvalues lie inside the unit circle, 
the equilibrium is locally asymptotically stable. Specifically, we use the sufficient and necessary condition for local stability, also called the Schur inequalities, 
\[
1-\operatorname{tr}J + \det J>0,\qquad 1+\operatorname{tr}J + \det J>0,\qquad 1-\det J>0.
\]
where $J$ is the Jacobian of \eqref{eq:closed-mu-s} \cite{elaydi2005}. 
\begin{figure}[h!]
    \centering

        \centering
        \includegraphics[width=0.8\textwidth]{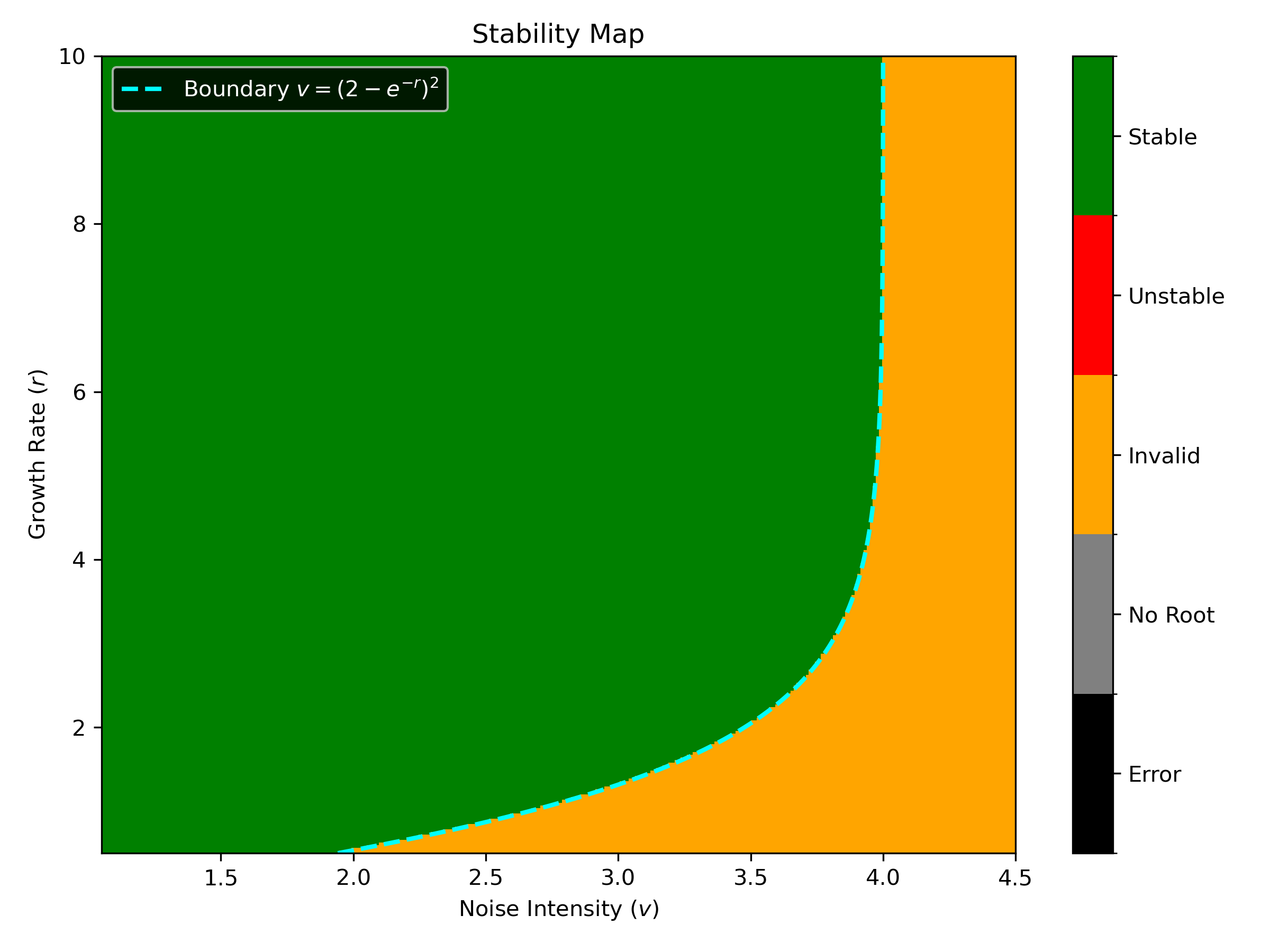}
        \caption{Stability Region}
        \label{map_2_stability}
     \label{fig:func212}
\end{figure}

While it is challenging to derive the closed form of the stability condition. In this section, we numerically demonstrate all the stability region of equilibria is the same as the existence region of equilibria in $(v,r)$. The numerical stability analysis was performed using a high-resolution grid search to map the system's dynamic regimes over the parameter space $r \in [0, 10]$ and $v \in [1.05, 4.5]$. The core procedure involved determining the equilibrium solution $z^*$ at each grid point and subsequently evaluating the local stability criterion. To ensure the validity of the simulation, the search range for $z^*$ was extended to $z_{\text{max}} = 30,000$. Since the theoretical bound $z < e^r - 1$ implies a maximum possible root of $e^{10} - 1 \approx 22,025$, this search limit strictly exceeds the theoretical upper bound, guaranteeing that the stability map captures the complete solution space without truncation artifacts. Figure \ref{map_2_stability} shows the stability classification of the Gamma--closure moment map in the $(r,v)$ parameter plane, 
The map was computed by solving the equilibrium condition $F(z;r,v)=0$ for $z=s/\mu$ and reconstructing the corresponding 
moment pair $(\mu^*,s^*)$; the Jacobian of the moment map was then evaluated to determine linear stability. 

A careful comparison of the resulting figures \ref{map_1_existence} and \ref{map_2_stability} reveal that the region of existence and the region of stability are identical. In the existence map, the blue region representing feasible equilibria ($\mu^* > 0$) is bounded sharply by the theoretical curve $v = (2 - e^{-r})^2$ in Figure \ref{map_1_existence}. Similarly, in the stability map in Figure \ref{map_2_stability}, the green region representing stable equilibria is confined by the exact same boundary. The area to the right of this boundary in the stability map is composed of ``Invalid'' points (where a mathematical root exists but yields unfeasible parameters such as $\mu^*\le 0\ \text{or}\ s^*\le 0.$). This precise overlap confirms that for this Gamma moment closure system, the threshold for the emergence of a positive population equilibrium coincides with the threshold for its local linear stability; effectively, wherever a biological equilibrium exists, it is stable.

\subsection{Monte--Carlo Simulation}
To evaluate the stability of the stochastic Ricker system, we performed a series of Monte--Carlo (MC) and moment--map simulations over an increasing sequence of growth rates \(r\in[0.5,3.0]\) with multiplicative noise intensity \(v \in [1, 3.0]\). In parallel, the deterministic moment equations 
\eqref{eq:closed-mu-s}
derived from the Gamma moment equations were iterated from the same initial moments.

Each Monte--Carlo experiment simulates an ensemble of independent stochastic Ricker trajectories,
\[
X_{n+1} = X_n\, e^{r(1 - X_n)}\,\varepsilon_n, \qquad t=0,1,\dots,T_{\max}-1,
\]
where \(r>0\) is the intrinsic growth rate and \(\varepsilon_n>0\) represents
multiplicative environmental noise.
The random factors \(\varepsilon_n\) are drawn from a \emph{lognormal} distribution
chosen so that the first two moments satisfy
\[
\mathbb{E}[\varepsilon_n]=1, \qquad \mathbb{E}[\varepsilon_n^2]=v>1.
\]
This ensures that the mean dynamics coincide with the deterministic Ricker map
when \(v=1\), while \(v>1\) introduces stochastic amplification of population fluctuations.
Specifically, we set
\(\varepsilon_n = \exp(\sigma Z_n - \tfrac{1}{2}\sigma^2)\)
with \(Z_n\sim \mathcal{N}(0,1)\) and \(\sigma^2 = \ln v\),
so that the multiplicative noise amplitude is fully determined by \(v\).

\begin{figure}[h!]
    \centering
    
    \begin{subfigure}[b]{0.45\textwidth}
        \centering
        \includegraphics[width=\textwidth]{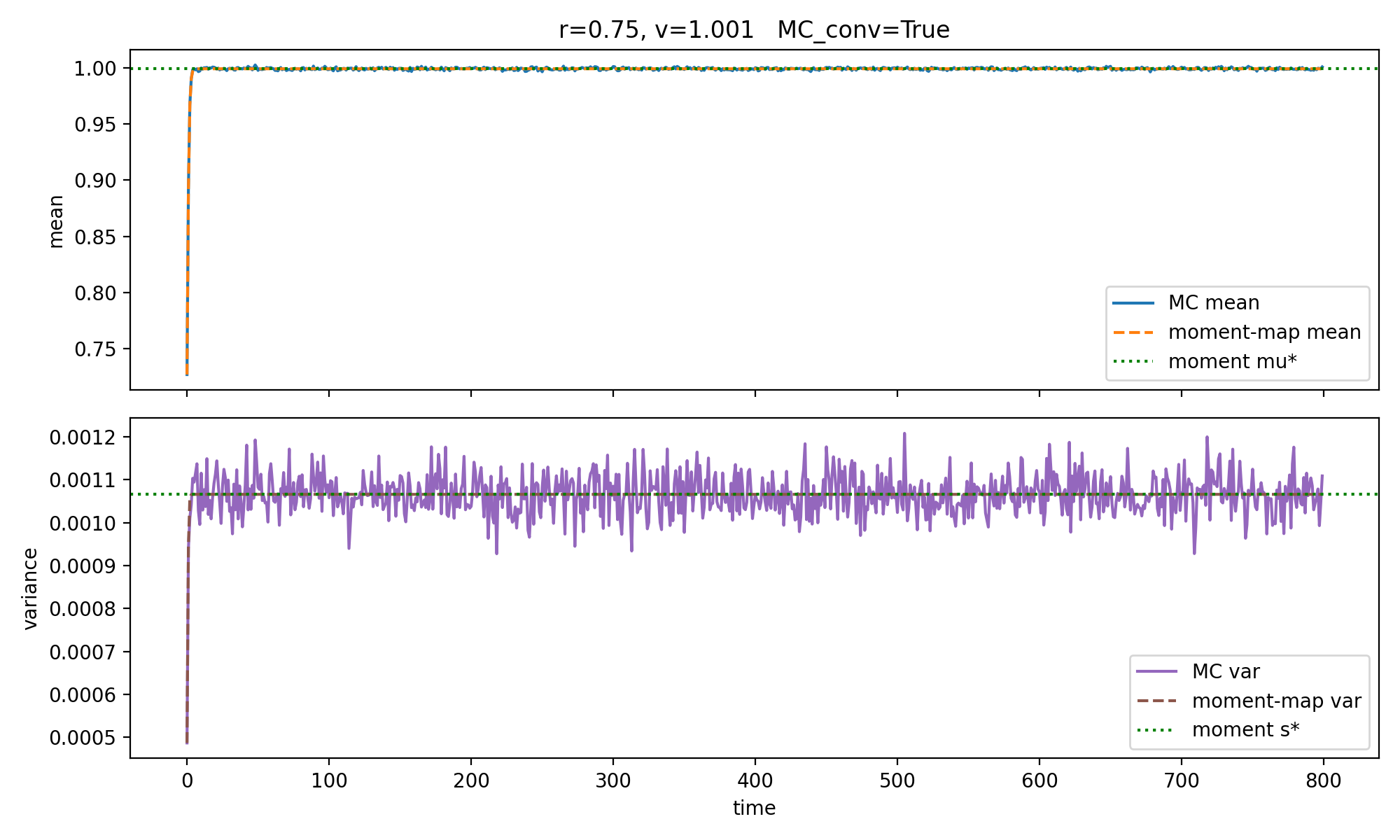}
        \caption{$r=0.75$, $v=1.001$}
        \label{fig:F1}
    \end{subfigure}
    \hfill
    \begin{subfigure}[b]{0.45\textwidth}
        \centering
        \includegraphics[width=\textwidth]{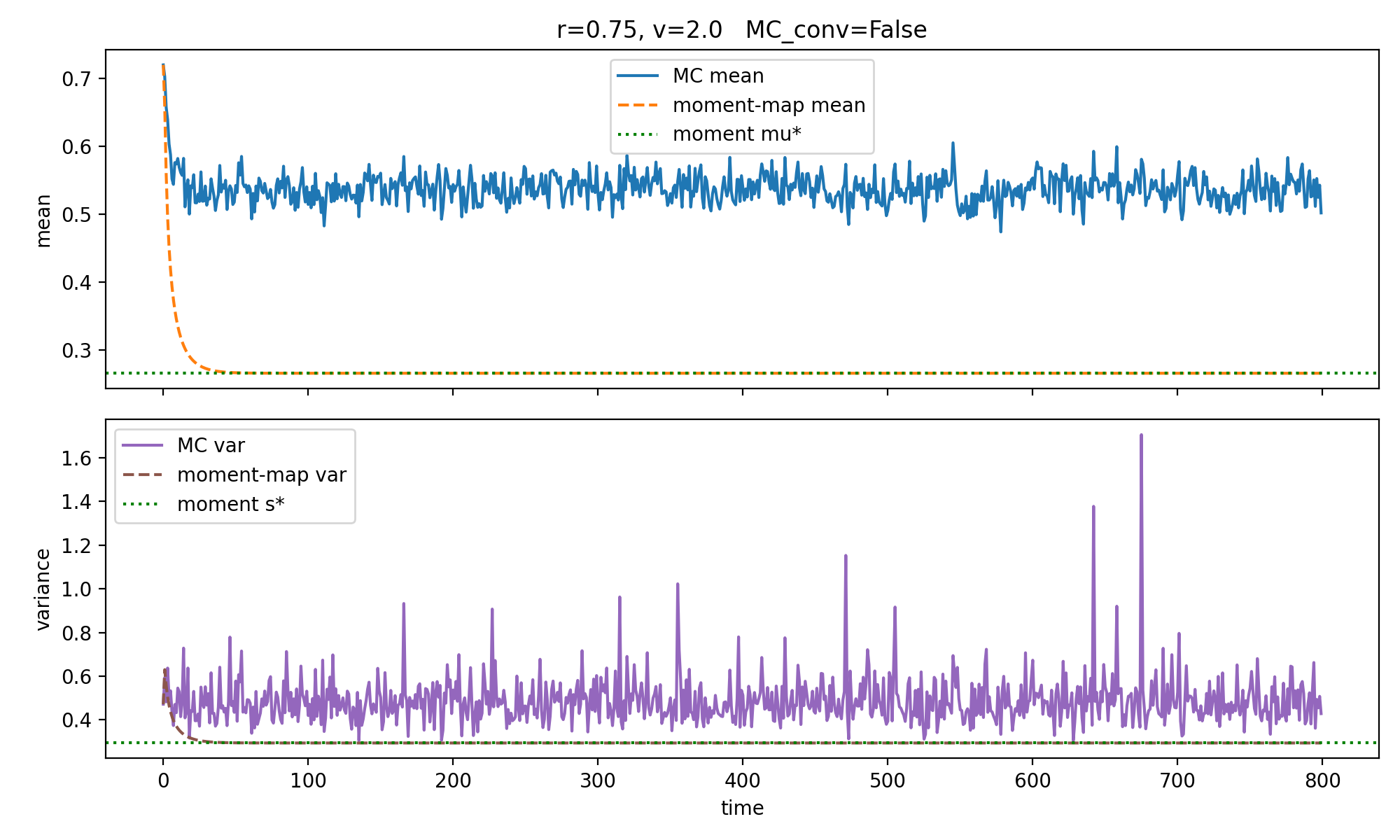}
        \caption{$r=0.75$, $v=2.0$}
        \label{fig:F2}
    \end{subfigure}
    
    \vspace{1em}  

    \begin{subfigure}[b]{0.45\textwidth}
        \centering
        \includegraphics[width=\textwidth]{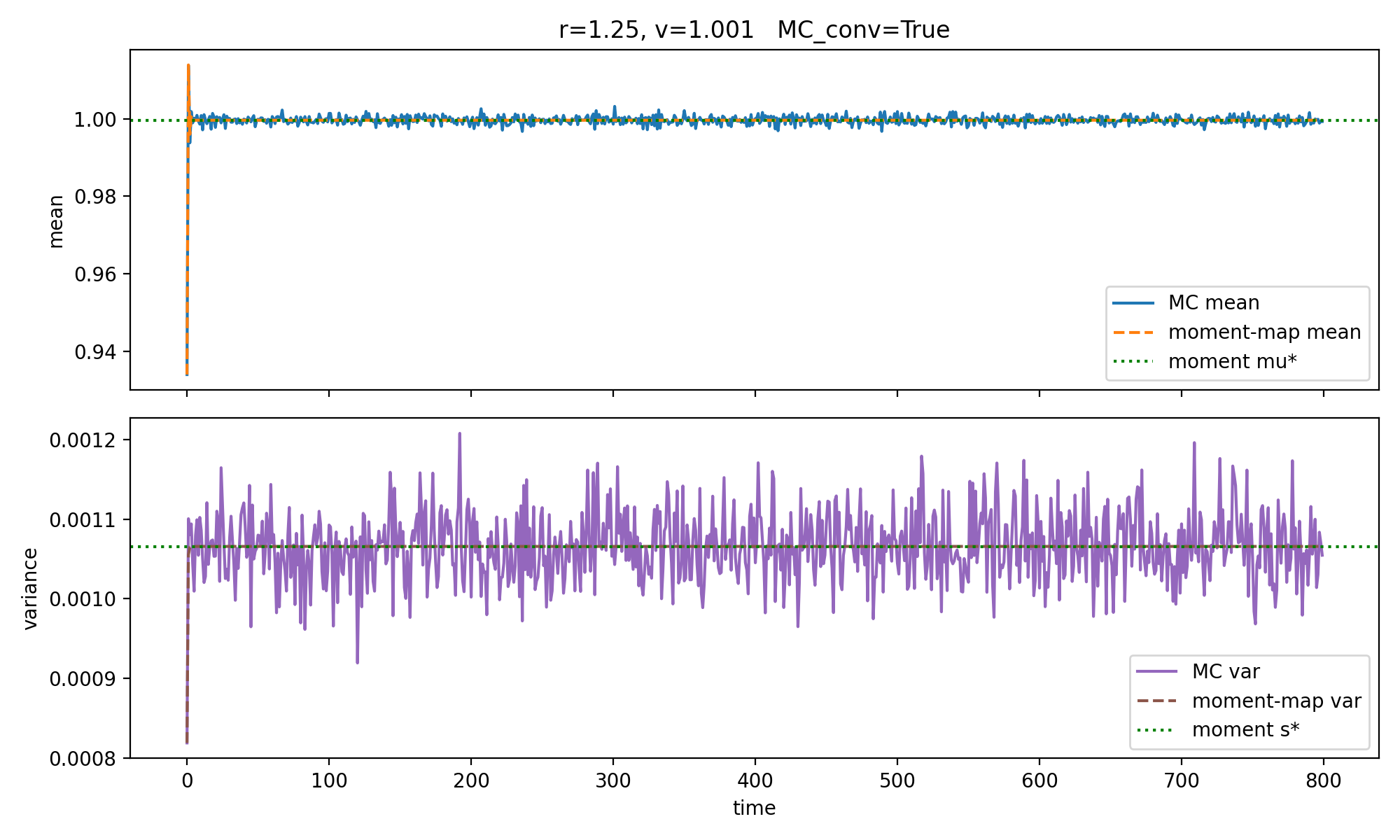}
        \caption{$r=1.25$, $v=1.001$}
        \label{fig:F3}
    \end{subfigure}
    \hfill
    \begin{subfigure}[b]{0.45\textwidth}
        \centering
        \includegraphics[width=\textwidth]{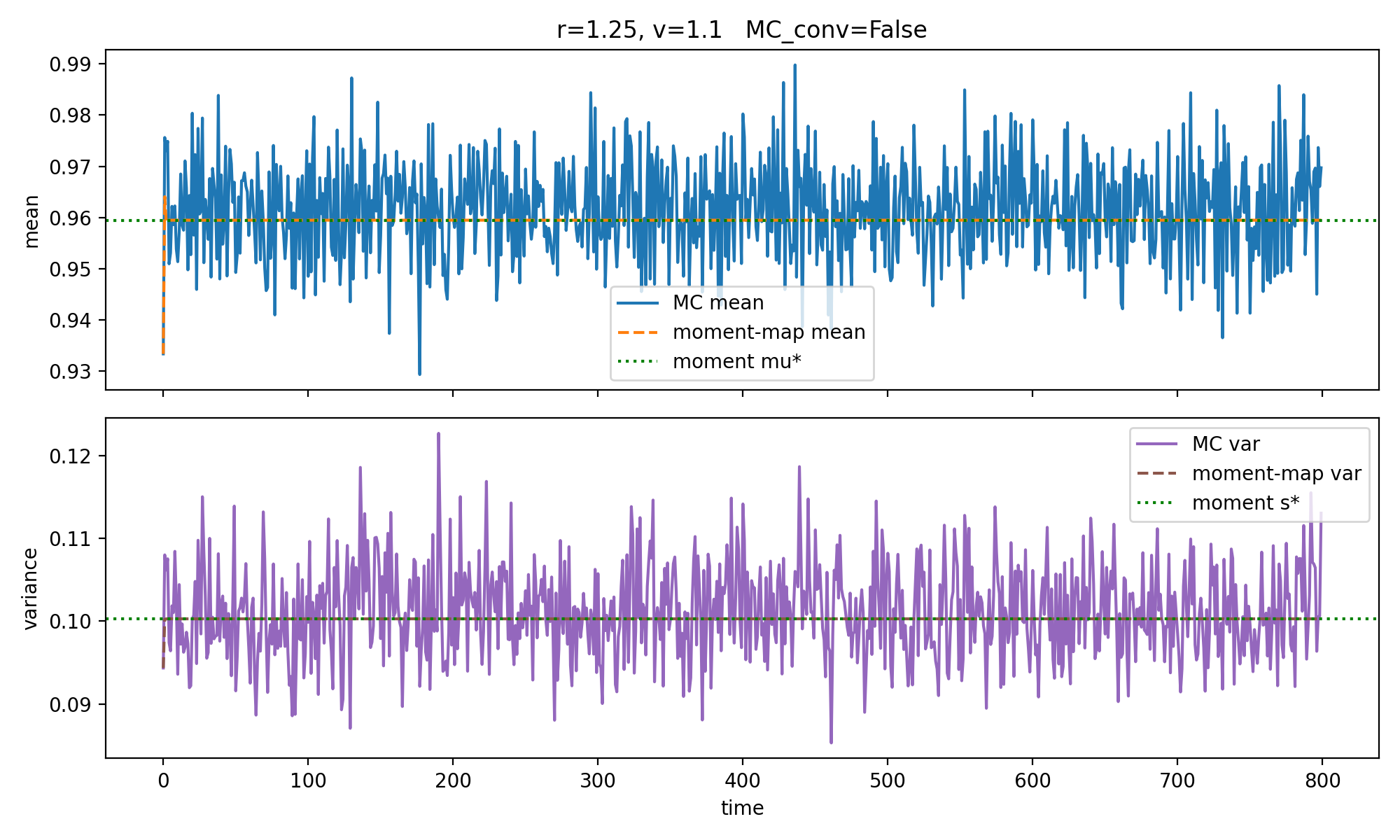}
        \caption{$r=1.25$, $v=1.1$}
        \label{fig:F4}
    \end{subfigure}
    
    \vspace{1em}  

    \begin{subfigure}[b]{0.45\textwidth}
        \centering
        \includegraphics[width=\textwidth]{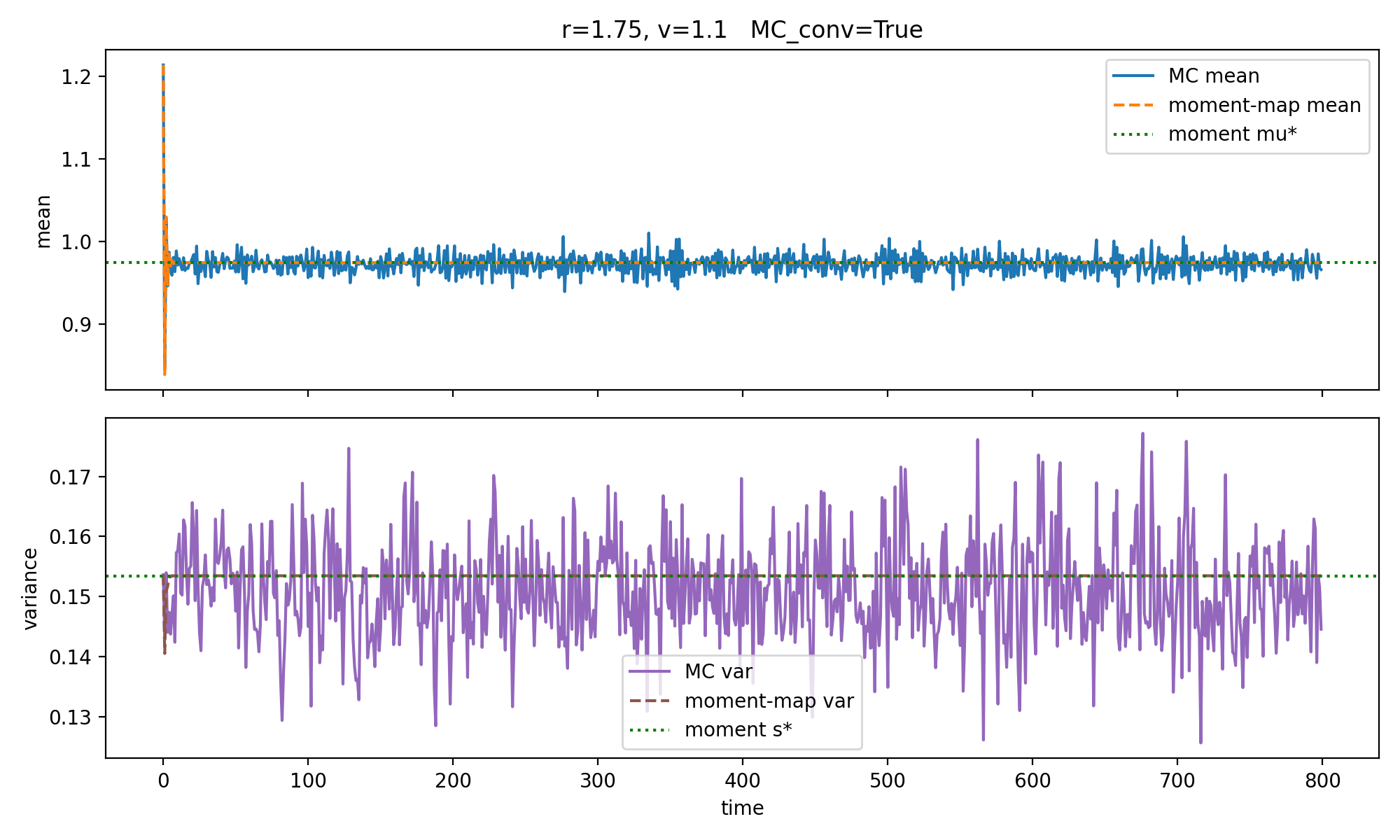}
        \caption{$r=1.75$, $v=1.1$}
        \label{fig:F5}
    \end{subfigure}
    \hfill
    \begin{subfigure}[b]{0.45\textwidth}
        \centering
        \includegraphics[width=\textwidth]{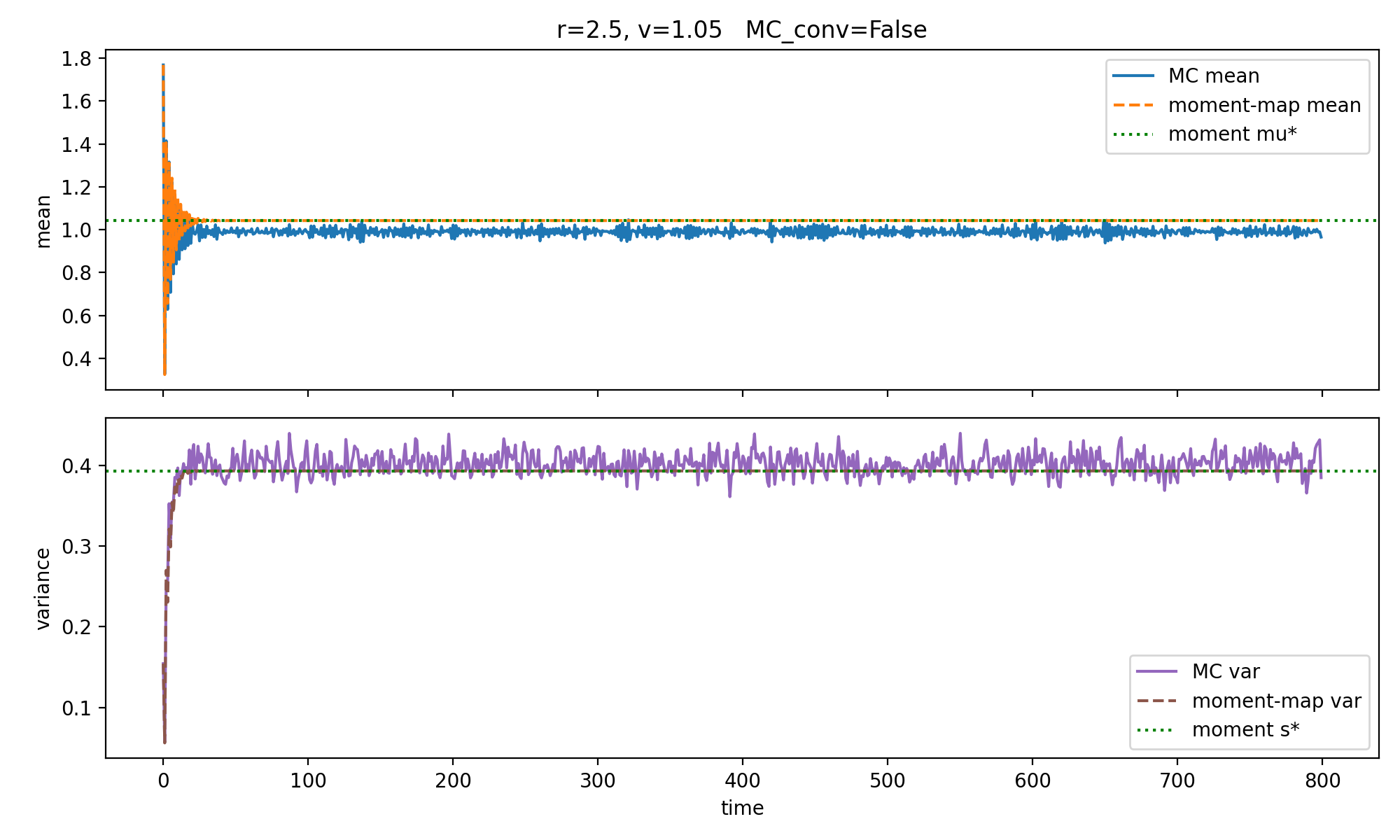}
        \caption{$r=2.5$, $v=1.05$}
        \label{fig:F6}
    \end{subfigure}
    
    \caption{Comparison of results for different values of $r$ and $v$.}
    \label{fig:carlo1}
\end{figure}

The initial population ensemble \(X_0^{(i)}\) for \(i=1,\dots,N_{\mathrm{ens}}\)
is drawn independently from a \emph{Gamma distribution}
whose mean and variance match the theoretical moment--map equilibrium
\((\mu^*,s^*)\) when it exists; otherwise, when no admissible equilibrium is available,
a narrow Gamma distribution centered near \(\mu_0=0.5\) with shape parameter \(k=10\)
and variance \(s_0=\mu_0^2/k\) is used.
This choice guarantees positive initial states and provides realistic dispersion
consistent with the closure assumption that population states remain Gamma--distributed.  For each parameter pair \((r,v)\), an ensemble of \(N_{\mathrm{ens}}\)
trajectories is evolved for \(T_{\max}\) iterations.
At every time step, the ensemble mean
\(\bar{X}_n = \frac{1}{N_{\mathrm{ens}}}\sum_i X_n^{(i)}\)
and variance are recorded.

For each parameter pair $(r,v)$, an ensemble of $N_{\mathrm{ens}}$ trajectories is evolved for $T_{\max}$ iterations. At every time step $n$, the ensemble mean $\bar{X}_n = \frac{1}{N_{\mathrm{ens}}}\sum_i X_n^{(i)}$ and variance are recorded. Convergence of the ensemble mean is declared if its coefficient of variation in the final observation window of length $w$ is less than $0.0001$. Mathematically, this requires satisfying:
\begin{equation}
\frac{\mathrm{std}\left(\{\bar{X}_n\}_{n=T_{\max}-w}^{T_{\max}}\right)}
     {\mathrm{mean}\left(\{\bar{X}_n\}_{n=T_{\max}-w}^{T_{\max}}\right)} < 0.0001,
\end{equation}
where the standard deviation and mean are calculated over the temporal window. If this condition is met, the run is considered converged to a statistically steady state; otherwise, it is marked as non-convergent.

Figures~ \ref{fig:carlo1} compare the time evolution of the ensemble mean (top panels) 
and variance (bottom panels) obtained from Monte--Carlo simulations (solid curves) 
and the deterministic moment--map predictions (dashed curves) for several combinations 
of growth rate \(r\) and noise intensity \(v\).
When the noise level is extremely small (\(v\approx1.001\)), the Monte--Carlo ensemble behaves 
almost deterministically and the ensemble mean rapidly stabilizes near the 
moment--map equilibrium. 
In this regime the convergence criterion based on a $1\%$ relative fluctuation of the mean 
is satisfied, and the system is classified as \emph{convergent} in Figure \ref{fig:F1}, \ref{fig:F3}  \ref{fig:F5}.
As \(v\) increases even slightly above unity (e.g.\ \(v=1.05\) or \(1.1\)), the multiplicative noise 
introduces persistent random fluctuations around the theoretical equilibrium. 
Although the long--term mean remains bounded and oscillates about the 
moment--map prediction, its relative standard deviation over the final observation window 
exceeds the $1\%$ threshold, so these runs are marked as \emph{non--convergent} in Figure \ref{fig:F2}, \ref{fig:F4}  \ref{fig:F6}.
For larger noise intensities (\(v=2\)), the stochastic forcing dominates: 
the ensemble mean and variance exhibit sustained variability and do not approach any 
steady value, while the deterministic moment equations still predict a stable equilibrium. 
This contrast highlights the breakdown of the moment closure approximation: while the closure assumes fluctuations are small enough to be smoothed out, strong multiplicative noise fundamentally alters the dynamics, preventing the system from settling into the stable equilibrium predicted by the model.

\section{Conclusions and Discussion}\label{discuss}

In this work, we developed a moment-based framework for analyzing stochastic population dynamics by deriving a closed system of difference equations for the Ricker model under the Gamma moment-closure approximation. We numerically assessed the validity of this approximation, demonstrating its efficacy in capturing the essential features of the original stochastic process. A primary theoretical contribution of this study is the establishment of the necessary and sufficient condition of the existence of the unique positive feasible equilibrium. By constructing a novel auxiliary function, we were able to overcome the algebraic complexity of the closed system and establish a solid analytical foundation for the moment dynamics.

From a biological perspective, the necessary and  sufficient condition, $v < (2 - e^{-r})^2$ provides explicit criteria for population persistence in fluctuating environments. This inequality marks a critical threshold: it reveals that while a higher intrinsic growth rate $r$ enhances stability, there is a fundamental limit to the environmental noise $v$ that a population can withstand. Our Monte Carlo simulations confirm these analytical predictions, illustrating how the interplay between growth rate and stochastic intensity determines stability and extinction risk. Collectively, these results offer a robust theoretical and computational tool for predicting population viability under environmental uncertainty.

As highlighted in Remark \ref{rem:limit}, the derived moment dynamical system recovers the classical deterministic Ricker map in the limit of vanishing noise ($v \to 1$) and the variance is sufficient small. Since the classical Ricker map is renowned for exhibiting a period-doubling route to chao, particularly for intrinsic growth rates $r > 2.69$ and displaying complex chaotic behavior for $r > 3$ \cite{may1976simple}. As a result, a natural direction for future research is a comprehensive bifurcation analysis of the closed moment system. While the current study focused on the existence and stability of the unique positive equilibrium, it remains an open question how the stochastic parameter $v$ interacts with high growth rates $r$ to influence these complex dynamics. Future work should investigate whether the introduction of environmental noise suppresses or shifts the onset of period-doubling bifurcations found in the deterministic limit, thereby providing a deeper understanding of how stochasticity impacts population predictability in the chaotic regime.

Although our numerical analysis strongly suggests that the local stability domain coincides exactly with the existence region defined by $v < (2-e^{-r})^2$, a rigorous analytical confirmation remains an open mathematical challenge. The algebraic complexity of the Jacobian matrix for the closed moment system currently prevents a direct derivation of the stability boundaries using the Schur inequalities. Future research should aim to overcome these technical hurdles, potentially through advanced algebraic simplification techniques or symbolic computation, to provide a complete theoretical proof that the existence of the positive feasible equilibrium implies its local stability. Establishing this analytical link would offer a fully self-contained mathematical theory for the stochastic Gamma-Ricker model without reliance on computational verification.

Despite the analytical tractability provided by the Gamma moment-closure technique, the approach relies on the fundamental assumption that the population distribution retains a unimodal, Gamma-like shape throughout its temporal evolution as highlighted by our numerical validation.  Consequently, this approximation may lose accuracy in dynamical regimes where the underlying stochastic process exhibits multi-modality, heavy tails, or complex transient behaviors that change significantly from the structural constraints of the Gamma family. To address these limitations, future research could extend the closure to higher-order moments or employ mixture models, such as a linear combination of Gamma distributions, to capture multi-modal phenomena.


\section{Declarations}
Competing Interests: The authors declare no competing interests..

\section*{Appendix}

\begin{figure}[h!]
    \centering
    
    \begin{subfigure}[b]{0.45\textwidth}
        \centering
        \includegraphics[width=\textwidth]{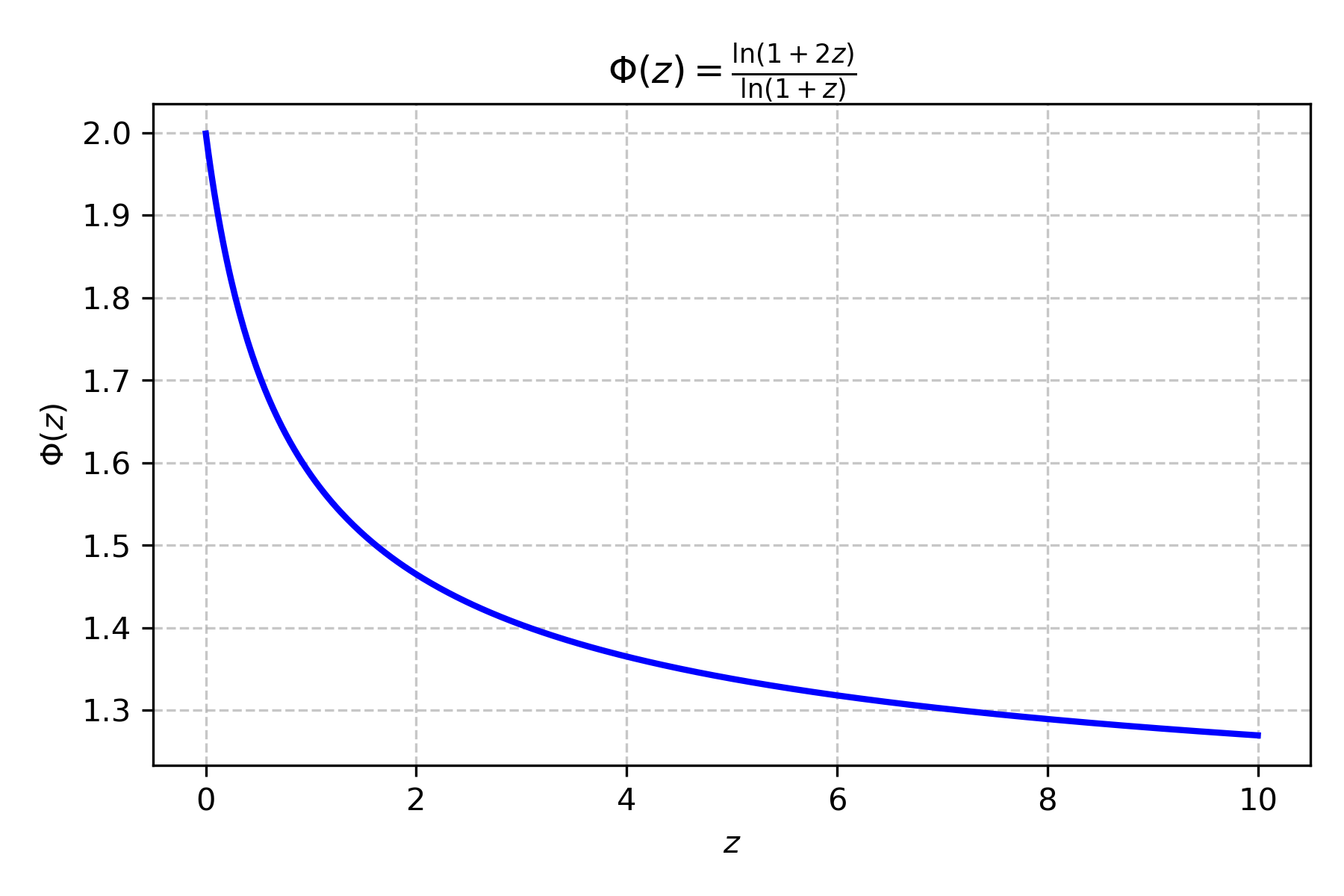}
        \caption{$\Phi(z)$ }
    \end{subfigure}
    \hfill
    \begin{subfigure}[b]{0.45\textwidth}
        \centering
        \includegraphics[width=\textwidth]{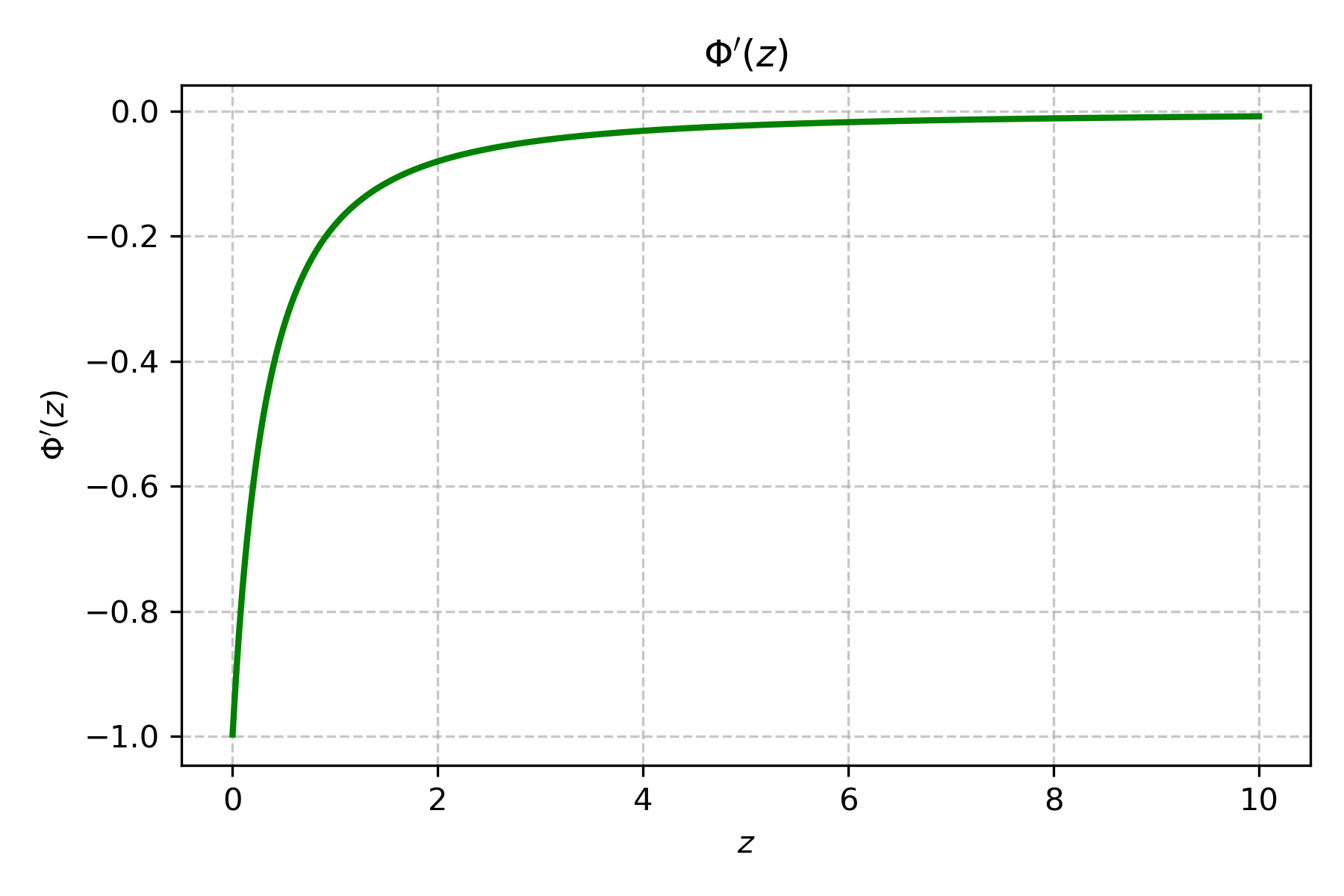}
        \caption{$\Phi'(z)$ }
   \end{subfigure}
   
   \vspace{1em}  

    \begin{subfigure}[b]{0.45\textwidth}
        \centering
        \includegraphics[width=\textwidth]{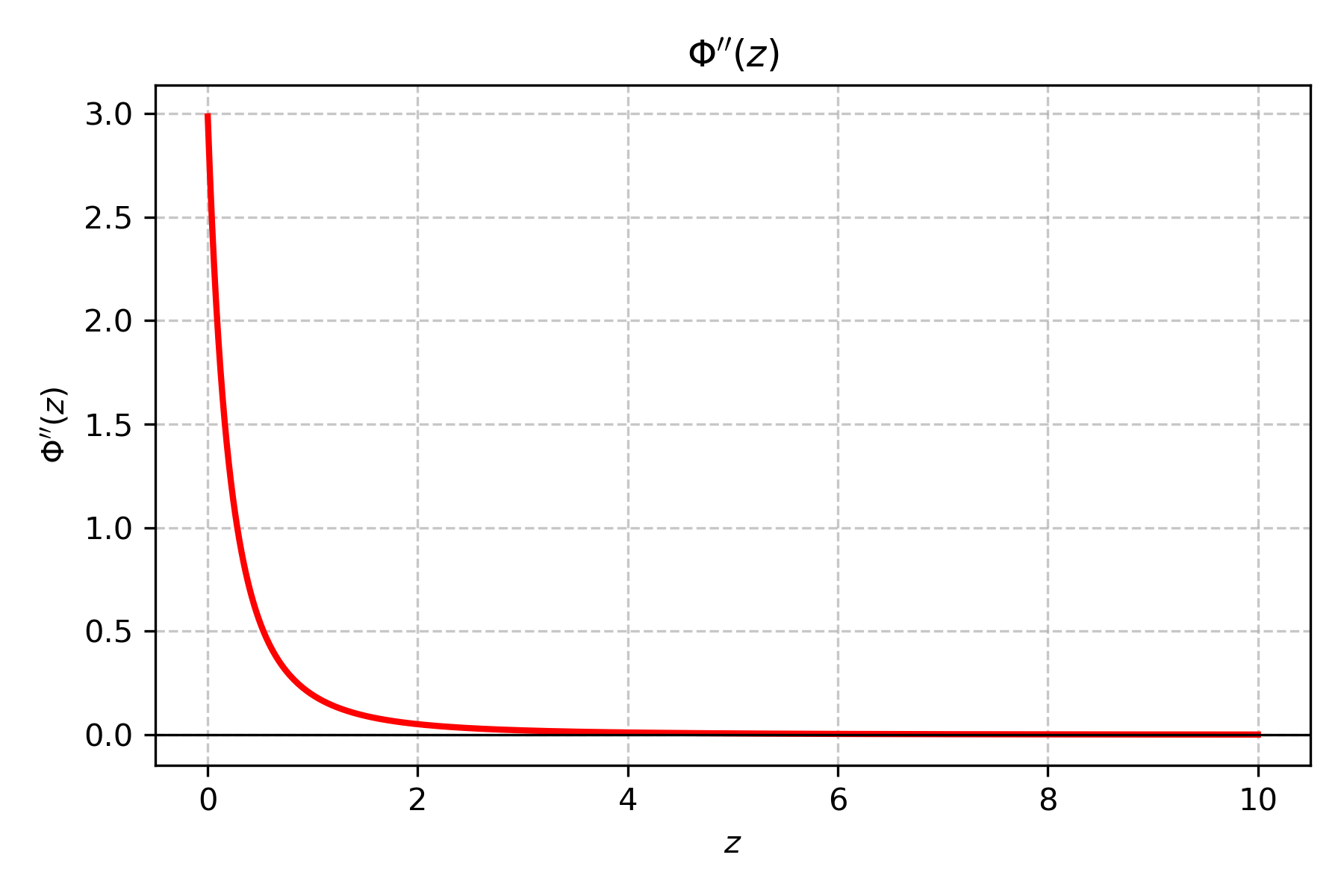}
        \caption{$\Phi''(z)$ }
    \end{subfigure}
    \hfill

   \vspace{1em}  
  \caption{Graphs of $\Phi(z)=\frac{\ln(1+2z)}{\ln(1+z)},\Phi'(z),\Phi''(z)$}
     \label{fig:phi}
\end{figure}

\begin{figure}[h!]
    \centering
    \begin{subfigure}[b]{0.45\textwidth}
        \centering
        \includegraphics[width=\textwidth]{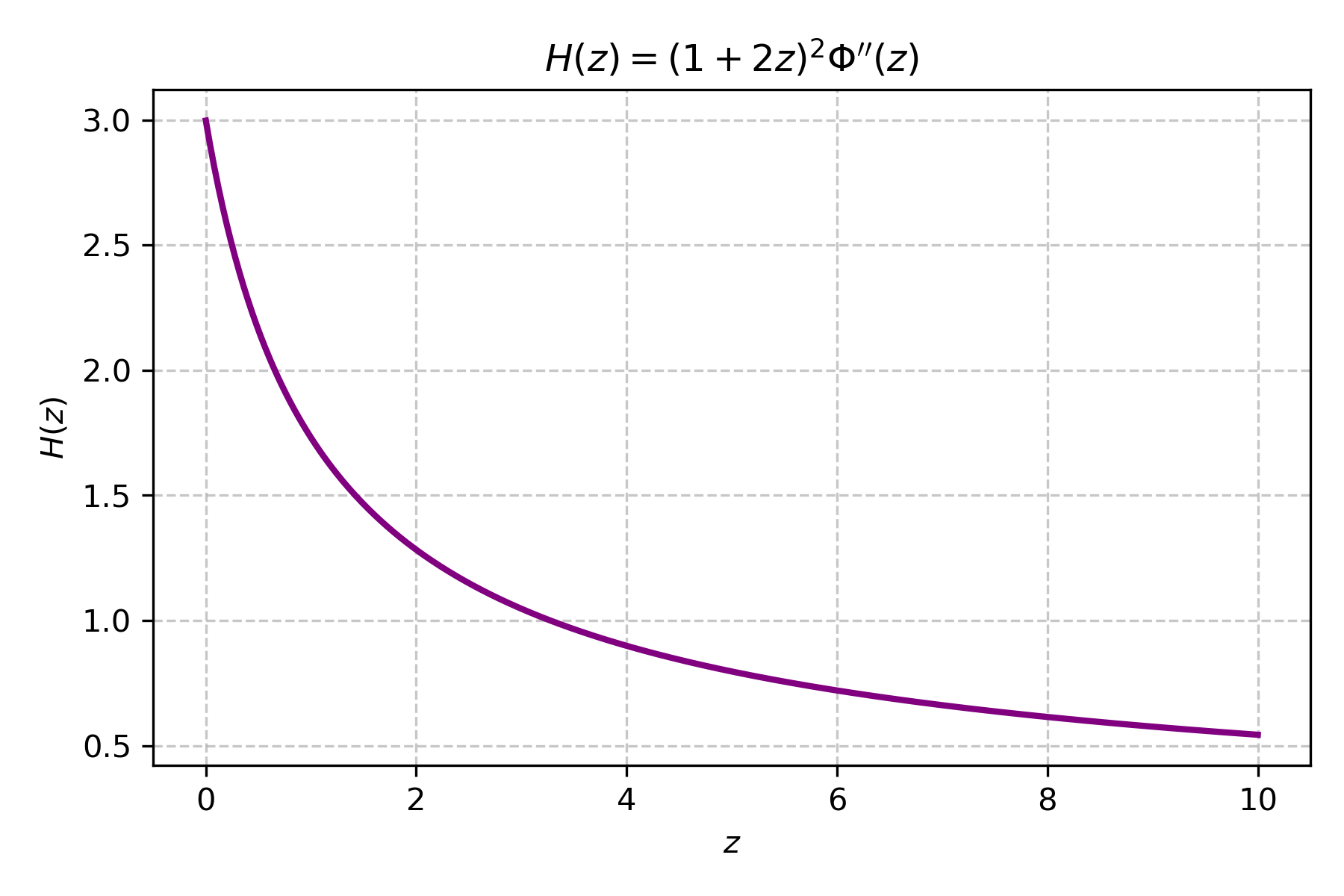}
        \caption{$H(z)$ }
        \label{fig:h3}
    \end{subfigure}
    \hfill
    \begin{subfigure}[b]{0.45\textwidth}
        \centering
        \includegraphics[width=\textwidth]{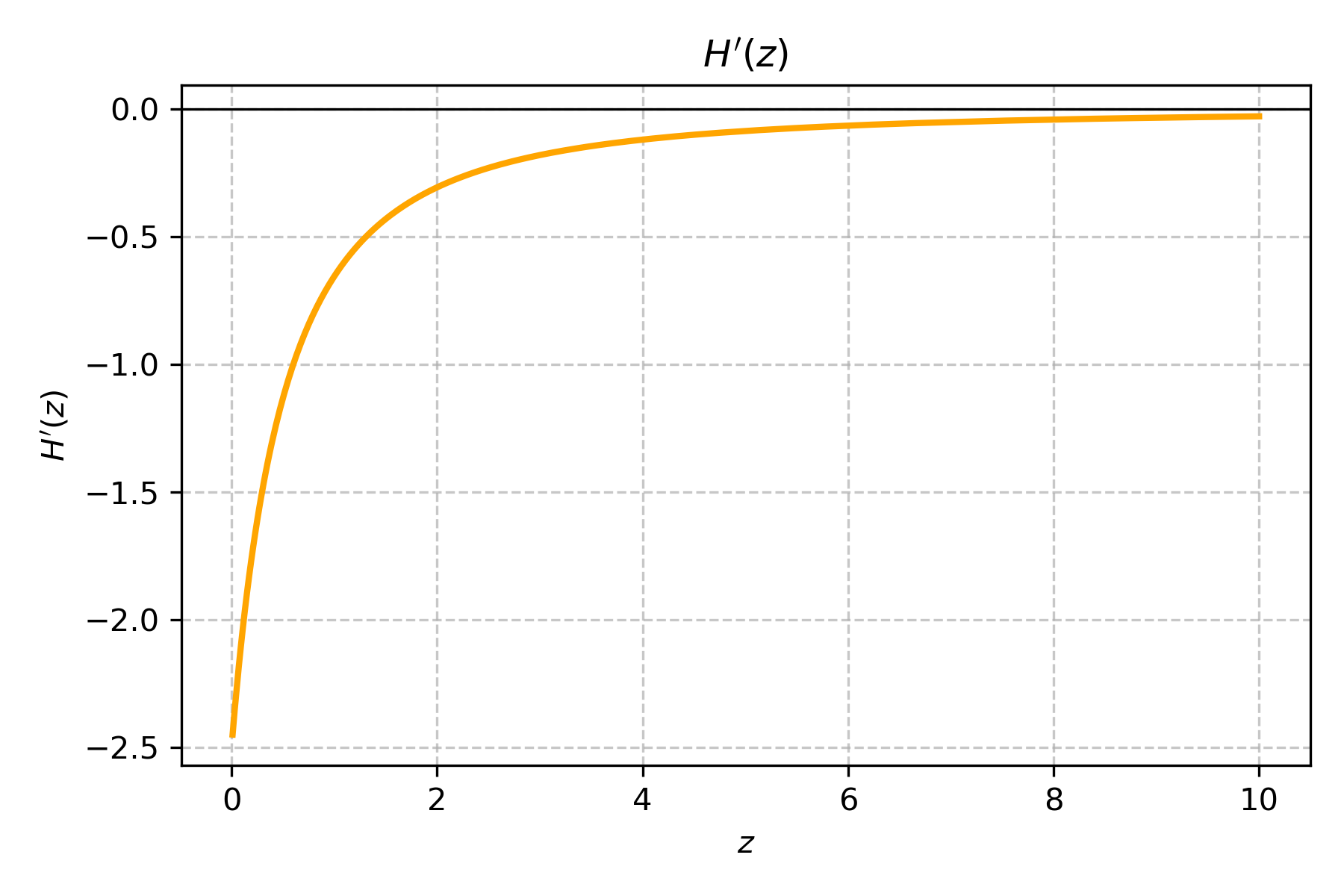}
        \caption{$H'(z)$ }
    \end{subfigure}

  \caption{Graphs of $H(z)=(1+2z)^2\Phi''(z), H'(z)$}
     \label{fig:H}
\end{figure}

\begin{lemma}\label{Phi}
 Define
\begin{equation}
  \Phi(z)\equiv\frac{\ln(1+2z)}{\ln(1+z)},\qquad H(z) \equiv \Phi''(z)(1+2z)^2,\qquad z>0
\end{equation}  
Then $\lim_{z \to 0} \Phi(z)=2$,  $\lim_{z \to 0} \Phi'(z)=-1$, $\lim_{z \to \infty} \Phi'(z) = 0$, $\lim_{z \to 0} H(z)=3$, $\lim_{z \to \infty} H(z) = 0$, $H'(z)<0$ for $z>0$.
\end{lemma}

\begin{proof}
The graphs of $\Phi(z)$, $\Phi'(z),\Phi''(z) $ are presented in Figures \ref{fig:phi}, and the graphs of $H(z)$, $H'(z)$ are shown in Figures \ref{fig:H} . The functions $\Phi(z)$ and $H(z)$ are elementary functions composed of logarithms and rational polynomials. Although they exhibit indeterminate forms at $z=0$ (specifically $0/0$), these singularities are removable. Thus, the functions are well-behaved: they are analytic on $z > 0$ and possess finite, well-defined limits at both $z=0$ and $z \to \infty$.  Let $$g(z) = \ln(1+z)$$ and $$h(z) = \ln(1+2z)$$ and have 
\[
\Phi(z) = \frac{h(z)}{g(z)}.
\]
The derivatives of the logarithmic terms are:
\[
\begin{aligned}
g'(z) &= \frac{1}{1+z}, & h'(z) &= \frac{2}{1+2z}, \\
g''(z) &= \frac{-1}{(1+z)^2}, & h''(z) &= \frac{-4}{(1+2z)^2}
\end{aligned}
\]
and
\[
\begin{aligned}
g'''(z) &= \frac{2}{(1+z)^3}, & h'''(z) &= \frac{16}{(1+2z)^3}.
\end{aligned}
\]

Using the quotient rule on $\Phi(z) = h(z)/g(z)$, we have
\[
\Phi'(z) = \frac{\frac{2}{1+2z}g(z) - \frac{1}{1+z}h(z)}{g(z)^2} = \frac{2(1+z)g(z) - (1+2z)h(z)}{(1+z)(1+2z)g(z)^2}.
\]
To compute $\Phi''(z)$, we differentiate the standard quotient form $\Phi' = (h'g - hg')/g^2$ again using the quotient rule:
\[
\Phi''(z) = \frac{g(z)[h''(z)g(z) - h(z)g''(z)] - 2g'(z)[h'(z)g(z) - h(z)g'(z)]}{g(z)^3}.
\]
Substituting the specific forms of $g', h', g'', h''$ and simplifying by clearing the denominators $(1+z)^2$ and $(1+2z)^2$, we obtain:
\begin{equation}\label{Phi-second}
\Phi''(z) = \frac{N(z)}{(1+z)^2 (1+2z)^2 g(z)^3},
\end{equation}
where $N(z)$ is given by:
\[
N(z) = (1+2z)^2 (g(z)+2)h(z) \;-\; 4(1+z)\big[(1+z)g(z) + (1+2z)\big]g(z).
\]
By definition $H(z) = (1+2z)^2 \Phi''(z)$, we have
\begin{equation}\label{H-form}
H(z) = \frac{(1+2z)^2 (g(z)+2)h(z) - 4(1+z)\big[(1+z)g(z) + (1+2z)\big]g(z)}{(1+z)^2 g(z)^3}.
\end{equation}
Differentiating $H(z)$ gives the following explicit expression:
\begin{equation}\label{Hprime-form}
H'(z) = -\frac{Q(z)}{(1+z)^3 (1+2z) g(z)^4},
\end{equation}
where the numerator $Q(z)$ is given by the following polynomial-logarithmic expression:
\[
\begin{aligned}
    Q(z)   &=  2(1+2z)^2 h(z) \Big[ 3(1+2z) + (2z-1)g(z) - g(z)^2 \Big] \\
          &\quad - 2(1+z)(1+2z) g(z) \Big[ 6(1+2z) + (1+4z)g(z) \Big]
\end{aligned}
\]
with $g(z) = \ln(1+z)$ and $h(z) = \ln(1+2z)$.

For the sake of brevity, we skip the detailed algebraic derivation and provide the resulting closed forms for $H(z)$ and $H'(z)$. To ensure validity, we cross-validated these derived formulas against the symbolic computation of $H(z)$ and $H'(z)$ based directly on the definition $(1+2z)^2 \Phi''(z)$ using computer algebra systems. The results, shown in Figure \ref{fig:MatchH}, confirm an exact match. Crucially, this verification relies on exact algebraic manipulation rather than numerical approximation; the fact that the symbolic difference simplifies to zero serves as a rigorous, computer-assisted proof of the derivation.

\begin{figure}[h!]
    \centering
    \begin{subfigure}[b]{0.45\textwidth}
        \centering
        \includegraphics[width=\textwidth]{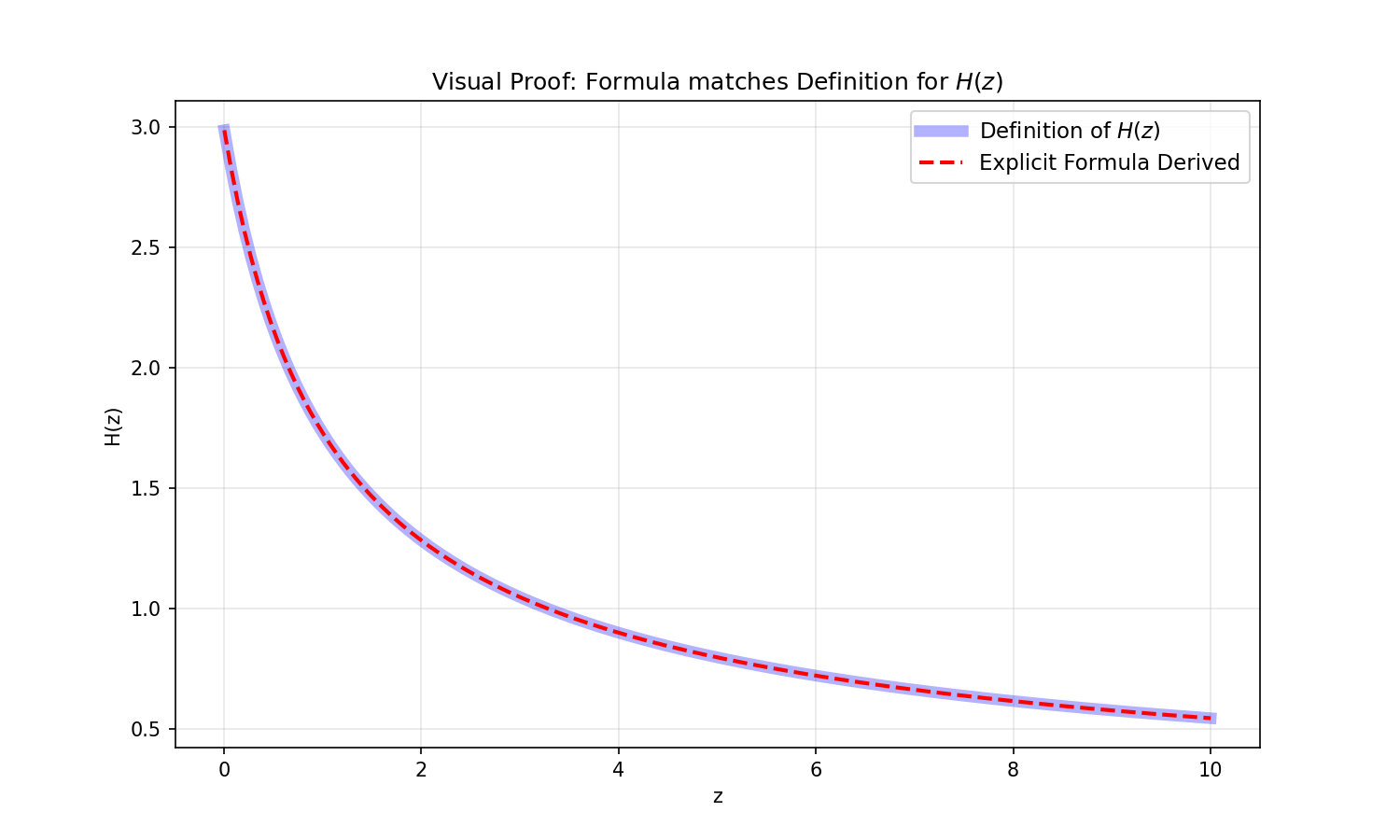}
        \caption{$H(z)$ }
        \label{fig:mahafdplot}
    \end{subfigure}
    \hfill
    \begin{subfigure}[b]{0.45\textwidth}
        \centering
        \includegraphics[width=\textwidth]{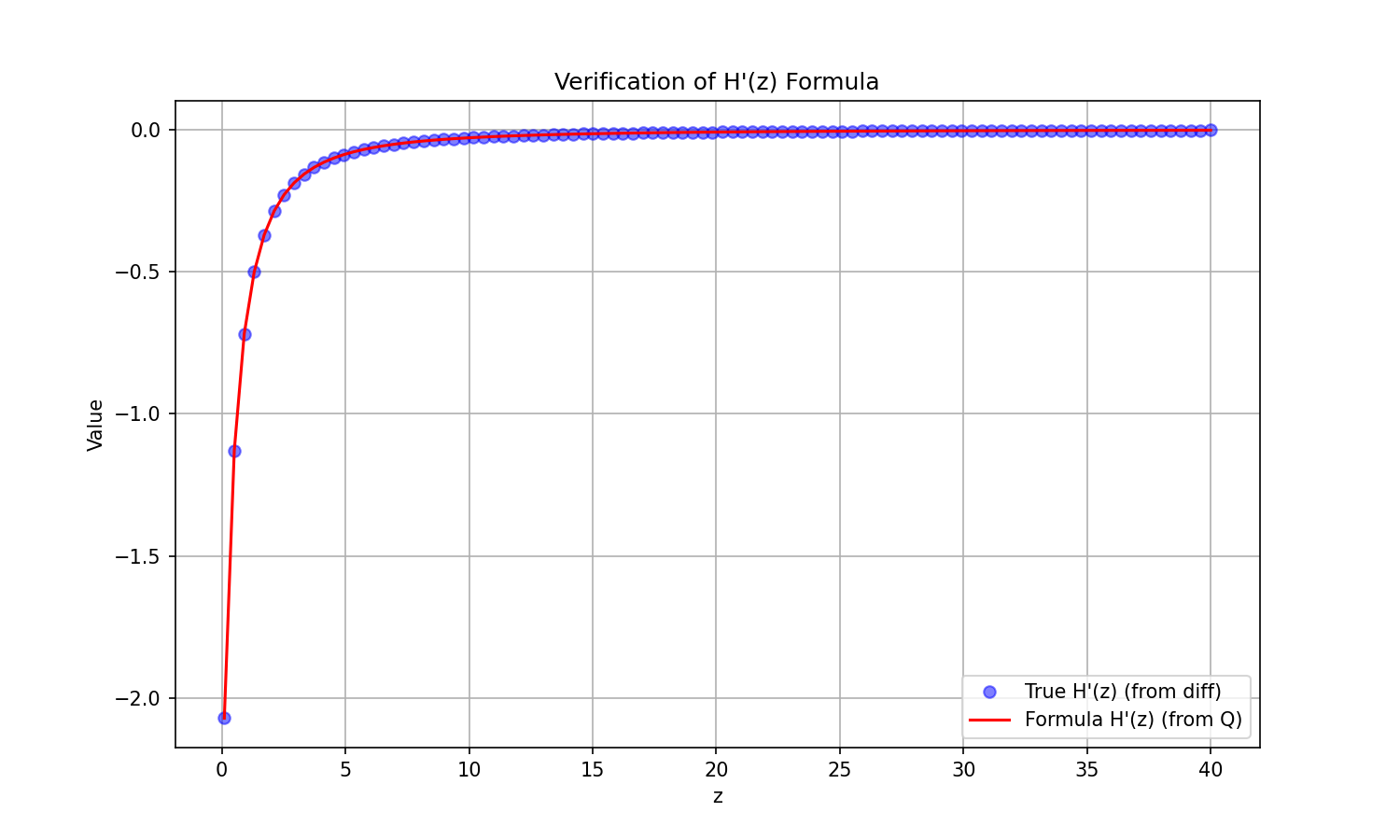}
        \caption{$H'(z)$ }
       \label{fig:mahplot}
    \end{subfigure}

  \caption{Validating the algebraic derivation of $H(z),H'(z)$ against Python symbolic computation}
     \label{fig:MatchH}
\end{figure}

\begin{figure}[h!]
    \centering
        \centering
        \includegraphics[width=\textwidth]{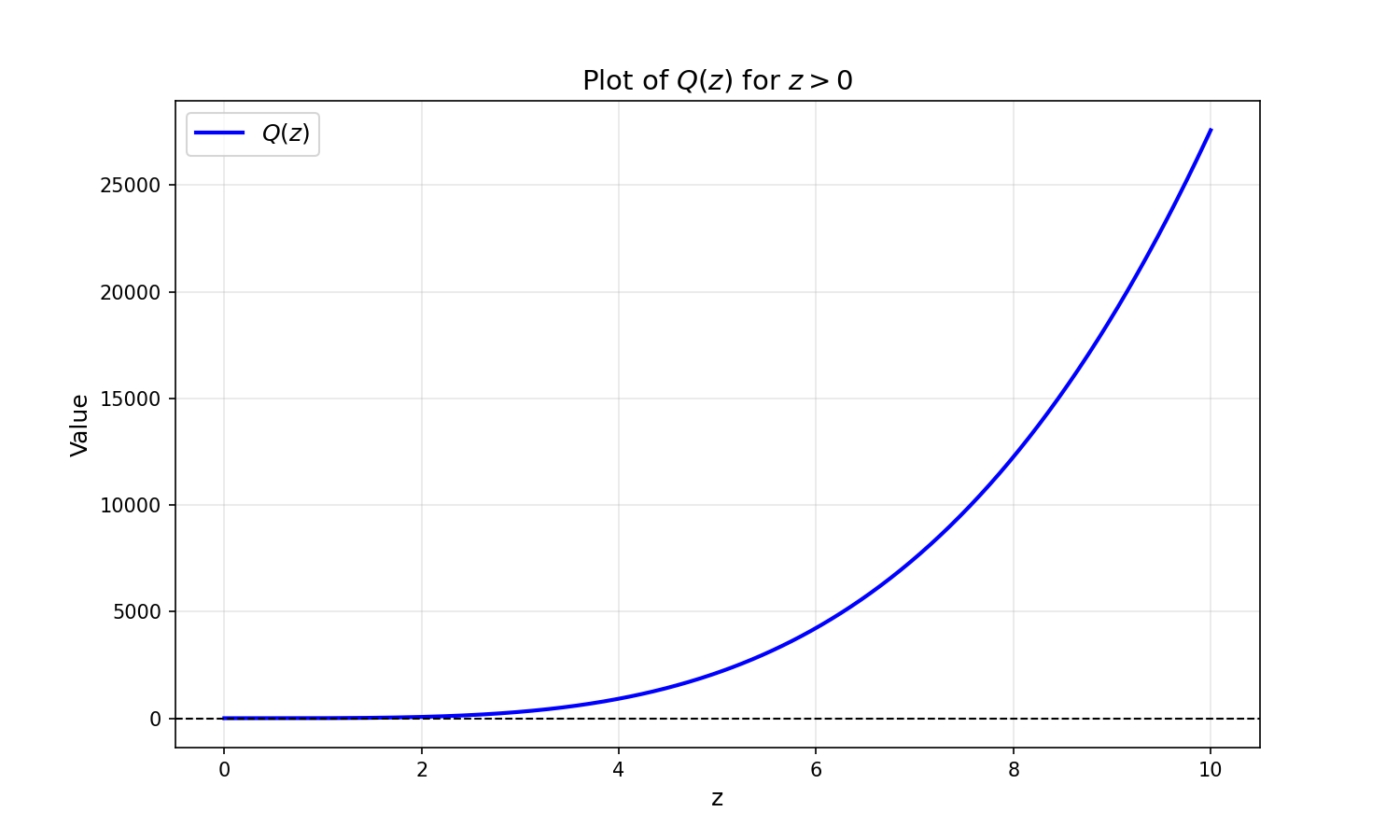}
        \caption{$Q(z)$ }
  \caption{Graph of $Q(z)$}
     \label{fig:Qz}
\end{figure}
Figure \ref{fig:mahafdplot} validates the algebraic derivation of $H(z)$ in \eqref{H-form}. The thick, semi-transparent blue line represents the ``Ground Truth,'' calculated by Python symbolically differentiating $\Phi(z)$ twice according to the definition $(1+2z)^2 \Phi''(z)$. The red line represents the values calculated using the derived fraction formula. The perfect superposition of the red dashes on the blue curve confirms that the derived formula is an exact algebraic equivalent of the definition. 

Figure \ref{fig:mahplot} validates the derived formula for $H'(z)$ in \eqref{Hprime-form}. The blue dotted line shows the derivative of the definition, computed symbolically by Python (effectively finding $\Phi'''$). The red line represents the values calculated using the derived fraction formula. The exact overlap confirms that the derivative of the formula behaves precisely as the definition requires. Since this curve is strictly negative, it confirms that $H(z)$ is monotonically decreasing and that the auxiliary function $Q(z)$ (which determines the sign of $H'$) is strictly positive in shown in Figure \ref{fig:Qz}.

Now we use the formula to compute the limits at $0$ and $\infty$, which are clearly demonstrated in Figures \ref{fig:phi},and Figures \ref{fig:H}. We formally employed L'Hôpital's rule to prove the limits at $z=0$ and $z \to \infty$.  Because the resulting analytic expressions for the higher derivatives become excessively lengthy, here we choose a simple approach to find these limits.  

First note that, for large $z$,  
\[
g(z) \sim \ln z, \quad h(z) \sim \ln z + \ln 2, \quad \text{and} \quad h(z) - g(z) \sim \ln 2.
\]
and derive the limits for $\Phi'(z)$ and $H(z)$ at $\infty$.

\begin{align}
\label{lim-Phi}
\lim_{z \to \infty} \Phi'(z) 
&= \lim_{z \to \infty} \frac{2(1+z)g(z) - (1+2z)h(z)}{(1+z)(1+2z)g(z)^2} \notag \\
&\sim \lim_{z \to \infty} \frac{2z g(z) - 2z h(z)}{2z^2 (\ln z)^2} 
= \lim_{z \to \infty} \frac{-2z [h(z)-g(z)]}{2z^2 (\ln z)^2} \notag \\
&= \lim_{z \to \infty} \frac{-\ln 2}{z (\ln z)^2} = 0. \\[10pt]
\label{lim-H}
\lim_{z \to \infty} H(z) 
&= \lim_{z \to \infty} \frac{(1+2z)^2 (g(z)+2)h(z) - 4(1+z)\big[(1+z)g(z) + (1+2z)\big]g(z)}{(1+z)^2 g(z)^3} \notag \\
&\sim \lim_{z \to \infty} \frac{4z^2 (g+2)h - 4z^2 (g+2)g}{z^2 (\ln z)^3} 
= \lim_{z \to \infty} \frac{4z^2 (g+2)[h(z)-g(z)]}{z^2 (\ln z)^3} \notag \\
&= \lim_{z \to \infty} \frac{4 (\ln z)(\ln 2)}{(\ln z)^3} 
= \lim_{z \to \infty} \frac{4 \ln 2}{(\ln z)^2} = 0.
\end{align}

At $z=0$, the required derivatives at $z=0$ are:
\begin{align*}
    g'(0) &= 1, & g''(0) &= -1, & g'''(0) &= 2, \\
    h'(0) &= 2, & h''(0) &= -4, & h'''(0) &= 16.
\end{align*}
Using L'Hôpital's rule:
\[
\Phi(0) = \lim_{z \to 0} \frac{h'(z)}{g'(z)} = \frac{2}{1} = 2.
\]
Differentiating $\Phi g = h$ twice gives $h'' = \Phi'' g + 2\Phi' g' + \Phi g''$. At $z=0$ ($g(0)=0$):
\[
h''(0) = 2\Phi'(0)g'(0) + \Phi(0)g''(0) \implies -4 = 2\Phi'(0)(1) + (2)(-1) \implies \Phi'(0) = -1.
\]
Differentiating a third time: $h''' = \Phi''' g + 3\Phi'' g' + 3\Phi' g'' + \Phi g'''$. At $z=0$:
\[
16 = 0 + 3\Phi''(0)(1) + 3(-1)(-1) + 2(2) \implies 16 = 3\Phi''(0) + 7 \implies \Phi''(0) = 3.
\]
Thus, $H(0) = \lim_{z \to 0} (1+2z)^2 \Phi''(z) = 3$.

\end{proof}

\end{document}